\documentclass[]{amsart}
\usepackage[hyperindex=true,bookmarks=true,bookmarksnumbered=true]{hyperref}

\usepackage{aliascnt}
\usepackage[inline]{enumitem}

\setlist[enumerate]{label=\textnormal{\bfseries(\alph*)}, leftmargin=*, nosep, widest=a}

\usepackage[all]{xy}
\allowdisplaybreaks

\makeatletter
\def\@seccntformat#1{%
  \protect\textup{%
    \protect\@secnumfont
    \expandafter\protect\csname format#1\endcsname 
    \csname the#1\endcsname
    \protect\@secnumpunct
  }%
}

\def\equationautorefname~#1\null{%
  Equation~(#1)\null
}

\newtheoremstyle{convenientthm}%
  {3pt}
  {3pt}
  {\itshape}
  {}
  {\bfseries}
  {.}
  {.5em}
  {\thmnumber{#2 }\thmname{#1}\thmnote{. #3\addcontentsline{toc}{subsection}{\tocsubsection {}{#2}{#1. #3}}}}

\newtheoremstyle{convenientplain}%
  {3pt}
  {3pt}
  {}
  {}
  {\bfseries}
  {.}
  {.5em}
  {\thmnumber{#2 }\thmname{#1}\thmnote{. #3\addcontentsline{toc}{subsection}{\tocsubsection {}{#2}{#1. #3}}}}

\theoremstyle{convenientthm}

\newaliascnt{theorem}{subsection}
\newtheorem{theorem}[theorem]{Theorem}
\aliascntresetthe{theorem}

\newtheorem*{theorem*}{Theorem}
\newtheorem*{lemma*}{Lemma}

\newaliascnt{lemma}{subsection}
\newtheorem{lemma}[lemma]{Lemma}
\aliascntresetthe{lemma}

\newaliascnt{proposition}{subsection}

\aliascntresetthe{proposition}

\newaliascnt{corollary}{subsection}
\newtheorem{corollary}[corollary]{Corollary}
\aliascntresetthe{corollary}

\theoremstyle{convenientplain}

\newaliascnt{remark}{subsection}
\newtheorem{remark}[remark]{Remark}
\aliascntresetthe{remark}

\def\X{\mathfrak X}
\def\al{\alpha}
\def\be{\beta}

\def\ep{\varepsilon}

\def\ka{\kappa}
\def\la{\lambda}

\def\si{\sigma}

\def\ph{\varphi}

\def\ps{\psi}
\def\om{\omega}
\def\Ga{\Gamma}
\def\De{\Delta}

\def\Ph{\Phi}

\def\i{^{-1}}
\def\x{\times}
\def\p{\partial}
\let\on=\operatorname

\def\AMSonly#1{}

\def\Id{\on{Id}}
\def\R{\mathbb R}

\def\Adj{\on{Adj}}
\def\Tr{\on{Tr}}
\def\vol{\on{vol}}
\def\Vol{\on{Vol}}
\def\Imm{\on{Imm}}
\def\g{\bar{g}}
\def\Diff{\on{Diff}}

\def\dim{\on{dim}}

\def\Met{\operatorname{Met}}
\def\na{\nabla}

\def\Fl{\operatorname{Fl}}

\title{Fractional Sobolev metrics on spaces of immersions}

\author[M.~Bauer, P.~Harms, P.~W.~Michor]
{Martin Bauer, Philipp Harms, Peter W.~Michor}

\address{Martin Bauer: Faculty for Mathematics, Florida State University, USA}
\email{bauer@math.fsu.edu}

\address{Philipp Harms: Freiburg Institute of Advanced Studies and Faculty for Mathematics, Freiburg University, Germany}
\email{philipp.harms@stochastik.uni-freiburg.de}

\address{
Peter W. Michor: Faculty for Mathematics, University of Vienna, Austria}
\email{peter.michor@univie.ac.at}

\address{All authors: Erwin Schr\"odinger Institut, Boltzmanngasse 9, 1090 Wien, Austria}

\date{\today}

\thanks{We thank Martins Bruveris, Boris Kolev, Andreas Kriegl, Peer Kunstmann, and Lutz Weis for helpful discussions. Moreover, we gratefully acknowledge support in the form of a Research in Teams stipend of the Erwin Schr\"odinger Institute Vienna. MB was partially supported by NSF-grant 1912037 (collaborative research in connection with NSF-grant 1912030). PH was partially supported in the form of a Junior Fellowship of the Freiburg Institute of Advanced Studies.}

\begin{document}

\begin{abstract}
We prove that the geodesic equations of all Sobolev metrics of fractional order one and higher on  spaces of diffeomorphisms and, more generally, immersions are locally well posed.
This result builds on the recently established real analytic dependence of fractional Laplacians on the underlying Riemannian metric.
It extends several previous results and applies to a wide range of variational partial differential equations, including the well-known Euler--Arnold equations on diffeomorphism groups as well as the geodesic equations on spaces of manifold-valued curves and surfaces.
\end{abstract}

\maketitle

\setcounter{tocdepth}{1}
\tableofcontents

\section{Introduction}
\label{sec:introduction} 

\subsection*{Background}

Many prominent partial differential equations (PDEs) in hydrodynamics admit variational formulations as geodesic equations on an infinite-dimensional manifold of mappings. 
These include the incompressible Euler \cite{Ar1966}, 
Burgers \cite{khesin2008geometry}, 
modified Constantin--Lax--Majda \cite{constantin1985simple, wunsch2010geodesic, bauer2016geometric}, Camassa--Holm \cite{camassa1993integrable, misiolek1998shallow, kouranbaeva1999camassa}, Hunter--Saxton \cite{hunter1991dynamics, lenells2007hunter}, surface quasi-geostrophic \cite{constantin1994formation, Was2016} and Korteweg--de Vries \cite{ovsienko1987korteweg} equations of fluid dynamics as well as the governing equation of ideal magneto-hydrodynamics \cite{vishik1978analogs, marsden1984semidirect}.
This serves as a strong motivation for the study of Riemannian geometry on mapping space. 
An additional motivation stems from the field of mathematical shape analysis, which is intimately connected to diffeomorphisms groups and other infinite-dimensional mapping spaces via Grenander's pattern theory \cite{grenander1998computational, younes2010shapes} and elasticity theory \cite{srivastava-klassen-book:2016, bauer2014overview}. 
 
The variational formulations allow one to study analytical properties of the PDEs in relation to geometric properties of the underlying infinite-dimensional Riemannian manifold \cite{shnirel1987geometry, misiolek2010fredholm, bauer2012vanishing, bauer2013geodesic, bauer2018vanishing, jerrard2019vanishing}.
Most importantly, local well-posedness of the PDE, including smooth dependence on initial conditions, is closely related to smoothness of the geodesic spray on Sobolev completions of the configuration space~\cite{EM1970}. 
This has been used to show local well-posedness of PDEs in many specific examples, cf.\@ the recent overview article \cite{Kol2017}.
An extension of this successful methodology to wider classes of PDEs requires an in-depth study of smoothness properties of partial and pseudo differential operators with non-smooth coefficients such as those appearing in the geodesic spray or, more generally, in the Euler--Lagrange equations. 
This is the topic of the present paper.

\subsection*{Contribution}

This article establishes local well-posedness of the geodesic equation for fractional order Sobolev metrics on spaces of diffeomorphisms and, more generally, immersions. 
A simplified version of our main result reads as follows:

\begin{theorem*}
On the space of immersions of a closed manifold $M$ into a Riemannian manifold $(N,\g)$, 
the geodesic equation of the fractional-order Sobolev metric
\begin{align*}
G_f(h,k) 
=
\int_M \g\big((1+\De^{f^*\g})^ph,k\big)\vol^{f^*\g},
\qquad
h,k \in T_f\Imm(M,N),
\end{align*}
is locally well-posed in the sense of Hadamard for any $p\in[1,\infty)$. 
\end{theorem*}

This follows from Theorems~\ref{thm:wellposed} and~\ref{thm:satisfies}. 
The result unifies and extends several previously known results:
\begin{itemize}
\item 
For integer-order metrics, local well-posedness on the space of immersions from $M$ to $N$ has been shown in \cite{Bauer2011b}.
However, the proof contained a gap, which was closed in \cite{Mueller2017} for $N=\mathbb R^n$, and which is closed in the present article for general $N$.
The strategy of proof, which goes back to Ebin and Marsden \cite{EM1970}, is to show that the geodesic spray extends smoothly to certain Sobolev completions of the space. 
Our generalization to fractional-order metrics builds on recent results about the smoothness of the functional calculus of sectorial operators \cite{bauer2018smooth}.

\item 
For $N=\mathbb R^n$, the set of $N$-valued immersions becomes a vector space, which simplifies the formulation of the geodesic equation; see \autoref{cor:flat}. 
The treatment of general manifolds $N$ requires a theory of Sobolev mappings between manifolds, which is developed in \autoref{sec:sobolev}. 
Moreover, in the absence of global coordinate systems for these mapping spaces, we recast the geodesic equation using an auxiliary covariant derivative following \cite{Bauer2011b}; see \autoref{lem:connection} and \autoref{thm:geodesic}.

\item 
For $M=N$ our result specializes to the diffeomorphism group $\Diff(M)$, which is an open subset of $\Imm(M,M)$. 
On $\Diff(M)$ we obtain local well-posedness of the geodesic equation for Sobolev metrics of order $p\in [1/2,\infty)$; see \autoref{cor:diffeos}. 
Analogous results have been obtained by different methods (smoothness of right-trivializations) for inertia operators that are defined as abstract pseudo-differential operators  \cite{escher2014right, bauer2015local,  BBCEK2019}.

\item 
For $M=S^1$, our result specializes to the space of immersed loops in $N$. 
For loops in $N=\mathbb R^d$, local well-posedness has been shown by different methods (reparameterization to arc length) in \cite{bauer2018fractional}.
Our analysis extends this result to manifold-valued loops and also to higher-dimensional and more general base manifolds $M$. 
\end{itemize}

\section{Sobolev mappings}
\label{sec:preliminaries}

\subsection{Setting}
\label{sec:setting}
We use the notation of \cite{Bauer2011b} and write $\mathbb N$ for the natural numbers including zero. 
Smooth will mean $C^\infty$ and real analytic $C^\om$.
Sobolev regularity is denoted by $H^r$, and Sobolev spaces $H^s_{H^r}$ of mixed order $r$ in the foot point and $s$ in the fiber are introduced in \autoref{thm:sobolev_mappings}.

Throughout this paper, without any further mention, we fix a real analytic connected closed manifold $M$ of dimension $\dim(M)$ and a real analytic manifold $N$ of dimension $\dim(N)\geq \dim(M)$.

\subsection{Sobolev sections of vector bundles}
\label{sec:sobolev}
\cite[Section~2.3]{bauer2018smooth}
We write $H^s(\mathbb R^m,\mathbb R^n)$ for the Sobolev space of order $s\in\mathbb R$ of $\mathbb R^n$-valued functions on $\mathbb R^m$. 
We will now generalize these spaces to sections of vector bundles.  
Let $E$ be a vector bundle of rank $n\in\mathbb N_{>0}$ over $M$. 
We choose a finite vector bundle atlas and a subordinate partition of unity in the following way. 
Let $(u_i\colon U_i \to u_i(U_i)\subseteq \mathbb R^m)_{i\in I}$ be a finite atlas for $M$, 
let $(\ph_i)_{i\in I}$ be a smooth partition of unity subordinated to $(U_i)_{i \in I}$, and let $\ps_i\colon E|U_i \to U_i\x \mathbb R^n$ be vector bundle charts. 
Note that we can choose open sets $U_i^\circ$ such that $\on{supp}(\ps_i)\subset U_i^\circ \subset \overline{U_i^\circ}\subset U_i$  and each $u_i(U_i^\circ)$ is an open set in $\mathbb R^m$ with Lipschitz boundary
(cf.\@ \cite[Appendix~H3]{behzadan2017certain}).
Then we define for each $s \in \mathbb R$ and $f \in \Ga(E)$
\begin{equation*}
\|f\|_{\Ga_{H^s}(E)}^2 := \sum_{i \in I} \|\on{pr}_cf{\mathbb R^n}\circ\, \ps_i\circ (\ph_i \cdot f)\circ u_i\i \|_{H^s(\mathbb R^m,\mathbb R^n)}^2.
\end{equation*}
Then $\|\cdot\|_{\Ga_{H^s}(E)}$ is a norm, which comes from a scalar product, and we write $\Ga_{H^s}(E)$ for the Hilbert completion of $\Ga(E)$ under the norm. 
It turns out that $\Ga_{H^s}(E)$ is independent of the choice of atlas and partition of unity, up to equivalence of norms. We refer to \cite[Section~7]{triebel1992theory2} and \cite[Section~6.2]{grosse2013sobolev} for further details. 

The following theorem describes module properties of Sobolev sections of vector bundles, which will be used repeatedly throughout the paper.

\begin{theorem}[Module properties]
\label{thm:module}
\cite[Theorem~2.4]{bauer2018smooth}
Let $E_1,E_2$ be vector bundles over $M$
and let $s_1,s_2,s\in\mathbb R$ satisfy  
$$s_1+s_2\geq 0,\; \min(s_1,s_2)\geq s, \text{ and } s_1+s_2-s>\dim(M)/ 2.$$
Then the tensor product of smooth sections extends to a bounded bilinear mapping 
\begin{equation*}
\Ga_{H^{s_1}}(E_1)\x \Ga_{H^{s_2}}(E_2) \to \Ga_{H^{s}}(E_1\otimes E_2).
\end{equation*}
\end{theorem}

The following theorem describes the manifold structure of Sobolev mappings between finite-dimensional manifolds.
It is an elaboration of \cite[5.2 and 5.4]{Michor20} and an extension to the Sobolev case of parts of \cite[Section 42]{KM97}.

\begin{theorem}[Sobolev mappings between manifolds]
\label{thm:sobolev_mappings}
The following statements hold for any $r\in(\dim(M)/2,\infty)$ and $s,s_1,s_2 \in [-r,r]$:
\begin{enumerate}
\item\label{thm:sobolev_mappings:a} 
The space $H^{r}(M,N)$ is a $C^{\infty}$ and a real analytic manifold. 
Its tangent space satisfies in a natural (i.e., functorial) way 
\begin{equation*}
T H^{r}(M,N) = H^r(M,TN) \xrightarrow[\pi_{H^{r}(M,N)}]{(\pi_N)_*} H^{r}(M,N)
\end{equation*} 
with foot point projection given by $\pi_{H^{r}}(M,N)=(\pi_N)_*\colon h \mapsto\pi_N\circ h$.
\item The space $H^s_{H^r}(M,TN)$ of `$H^s$ mappings $M\to TN$ with foot point in $H^{r}(M,N)$' is a real analytic manifold and a real analytic vector bundle over $H^{r}(M,N)$. 
Similarly, spaces such as  
$L(H^{s_1}_{H^r}(M,TN),H^{s_2}_{H^r}(M,TN))$ are real analytic vector bundles over $H^{r}(M,N)$.
\item The space $\Met_{H^{r}}(M)$ of all Riemannian metrics of Sobolev regularity $H^r$ is an open subset of the Hilbert space $\Ga_{H^r}(S^2T^*M)$, and thus a real analytic manifold.
\end{enumerate}
\end{theorem}

\begin{proof}
\begin{enumerate}[wide]
\item 
Let us recall the chart construction:
we use an auxiliary real analytic Riemannian metric $ \hat g$ on $N$ and its exponential mapping $\exp^{ \hat g}$; some of its  properties  are summarized in the following diagram:
$$\xymatrix
{
& 0_N \ar@{_{(}->}[d] \ar@(l,dl)[dl]_-{\text{zero section\quad }} 
& & N \ar@{^{(}->}[d] \ar@(r,dr)[rd]^-{\text{ diagonal}} & 
\\
TN &  V^{TN} \ar@{_{(}->}[l]^{\text{ open  }} \ar[rr]^{ (\pi_N,\exp^{ \hat g}) }_{\cong} & & V^{N\x N} 
\ar@{^{(}->}[r]_{\text{ open  }} & N\x N
}$$
Without loss we may assume that $V^{N\x N}$ is symmetric: 
$$(y_1,y_2)\in V^{N\x N} \iff (y_2,y_1)\in V^{N\x N}.$$
A chart, centered at a real analytic $f\in C^\om(M,N)$, is:
\begin{align*}
{H^r}(M,N)\supset U_f &=\{g: (f,g)(M)\subset V^{N\x N}\} \xrightarrow{u_f}{} \tilde U_f 
\subset \Ga_{H^r}( f^*TN)
\\
u_f(g) = (\pi_N,&\exp^{ \hat g})\i \circ (f,g),\quad u_f(g)(x) = (\exp^{ \hat g}_{f(x)})\i(g(x)) 
\\
(u_f)\i(s) &= \exp^{ \hat g}_f\circ s, \qquad (u_f)\i(s)(x) = \exp^{ \hat g}_{f(x)}(s(x))
\end{align*}
Note that $\tilde U_f$ is open in $\Ga(f^*TN)$. 
The charts $U_f$ for $f\in C^\om(M,N)$ cover $H^r(M,N)$: 
since $C^\om(M,N)$ is dense in $H^r(M,N)$ by \cite[42.7]{KM97} and since $H^r(M,N)$ is continuously embedded in $C^{0}(M,N)$, a suitable $C^0$-norm neighborhood of $g\in H^r(M,N)$ contains a real analytic $f \in C^\om(M,N)$, thus $f\in U_g$, and by symmetry of $V^{N\x N}$ we have $g\in U_f$.

The chart changes, 
$$
\Ga_{H^r}( f_1^*TN)\supset \tilde U_{f_1}\ni s  \mapsto (u_{f_2,f_1})_*(s) := 
(\exp^{ \hat g}_{f_2})\i\circ \exp^{ \hat g}_{f_1}\circ s\in \tilde U_{f_2}\subset \Ga_{H^r}( f_2^*TN),
$$
for charts centered on real analytic $f_1,f_2\in C^\om(M,N)$ are real analytic by Lemma \ref{lem:pushforward} since $r>\dim(M)/2$.

The tangent bundle $TH^r(M,N)$ is canonically glued from the following vector bundle chart changes, which are real analytic by Lemma \ref{lem:pushforward} again: 
\begin{multline}\label{eq:Tchartchanges}
\tilde U_{f_1}\x \Ga_{H^r}( f_1^*TN) \ni (s, h) \mapsto 
 (T (u_{f_2,f_1})_*)(s,h) =  
 \\
 =\big((u_{f_2,f_1})_*(s),  (d_{\text{fiber}} u_{f_2,f_1})_*(s,h)\big)\in \tilde U_{f_2}\x \Ga_{H^r}( f_2^*TN)
\end{multline}
It has the canonical charts 
$$
TH^r(M,N) \supset T\tilde U_f \xrightarrow[(T(\exp_f^{ \hat g})\i_*)]{Tu_f} \tilde U_f \x \Ga_{H^r}( f^*TN).
$$
These identify $TH^r(M,N)$ canonically with $H^r(M,TN)$ since
$$Tu_f\i(s,s') = T(\exp_f^{ \hat g})\circ\on{vl}\circ (s,s')\colon M\to TN\,,$$
where we used  the vertical lift $\on{vl}\colon TN\x_N TN \to TTN$ which is given by $\on{vl}(u_x,v_x)= \p_t|_{t=0}(u_x+t.v_x)$; see 
\cite[8.12 or 8.13]{Michor08}.  
The corresponding foot-point projection is then 
$$\pi_{H^s(M,N)}(T(\exp_f^{ \hat g})\circ \on{vl}\circ (s,s'))= \exp_f^{ \hat g}\circ s = \pi_N\circ T(\exp_f^{ \hat g})\circ (s,s').$$

\item The canonical chart changes (\ref{eq:Tchartchanges}) for $TH^r(M,N)$ extend to 
\begin{multline*}
\tilde U_{f_1}\x \Ga_{H^s}( f_1^*TN) \ni (s, h) \mapsto 
 (T u_{f_2,f_1})_*(s,h) =  
 \\
 =\big((u_{f_2,f_1})_*(s),  (d_{\text{fiber}} u_{f_2,f_1}\circ s)_*(h)\big)\in \tilde U_{f_2}\x \Ga_{H^s}( f_2^*TN),
\end{multline*}
since $d_{\text{fiber}}u_{f_2,f_1}\colon f_1^*TN\x_M f_1^*TN = f_1^* (TN\x_N TN)\to f_2^* TN$ is fiber respecting real analytic by the module properties \ref{thm:module}. 
Note that $d_{\text{fiber}}u_{f_2,f_1}\circ s$ is then an $H^r$-section of the bundle $L(f_1^*TN,f_1^*TN)\to M$, which may be applied to the $H^s$-section $h$ by the module properties \ref{thm:module}. 
These extended chart changes then glue the vector bundle 
$$H^s_{H^r}(M,TN)\xrightarrow{(\pi_N)_*} H^r(M,TN). $$ 

\item The space $\Gamma_{H^r}(S^2T^*M)$ is continuously embedded in $\Ga_{C^1}(S^2T^*M)$ because $r>\dim(M)/2+1$. Thus, the space of metrics is open.
\qedhere
\end{enumerate}
\end{proof}

\subsection{Connections, connectors, and sprays}
\label{sec:connection}
This sections reviews some relations between connections, connectors, and sprays.
It holds for general convenient manifolds $N$, including infinite-dimensional manifolds of mappings, and will be used in this generality in the sequel (see e.g.\@ the proofs of Theorems \ref{thm:geodesic} and \ref{thm:wellposed}). 
\begin{enumerate}
\item \textbf{Connectors.} 
\label{sec:connection:a}
\cite[22.8--9]{Michor08} Any connection $\nabla$ on $TN$ is given in terms of a connector $K\colon TTN\to TN$ as follows:
For any manifold $M$ and function $h\colon M\to TN$, one has $\nabla h = K \circ Th\colon TM\to TN$. 
In the subsequent points we fix such a connection and connector on $N$.
\item \textbf{Pull-backs.} 
\label{sec:connection:b}
\cite[(22.9.6)]{Michor08}
For any manifold $Q$, smooth mapping $g\colon Q\to M$ and $Z_y\in T_yQ$, one has $\nabla_{Tg.Z_y}s = \nabla_{Z_y}(s\circ g)$.
Thus, for $g$-related vector fields $Z\in \X(Q)$ and $X\in \X(M)$, one has $\nabla_Z(s\circ g) = (\nabla_Xs)\circ g$, as summarized in the following diagram:
\begin{equation*}
\xymatrix@C=4em @R=.3em{
& & T^2N \ar[dd]^{K} \\
& & \\
TQ \ar[r]_{Tg} \ar[rruu]^{T(s\circ g)} & TM \ar[ruu]_{Ts} & TN \\
& & TN \ar[dd]^{\pi_N} \\
Q \ar[r]^{g} \ar[uu]^{Z} & M \ar[uu]^{X} \ar[ruu]^{\nabla_Xs} & \\
Q \ar[r]^{g} & M \ar[ruu]^{s} \ar[r]^{f} & N. \\
}\end{equation*}

\item \textbf{Torsion.} 
\label{sec:connection:c}
\cite[(22.10.4)]{Michor08}
For any smooth mapping $f\colon M\to N$ and vector fields $X,Y\in \X(M)$ we have
\begin{align*}
\on{Tor}(Tf.X,Tf.Y)
     &=\nabla_X(Tf\circ Y) - \nabla_Y(Tf\circ X) - Tf\circ [X,Y] 
\\&
= (K \circ \ka_M - K) \circ TTf\circ TX \circ Y.
\end{align*}

\item \textbf{Sprays.} 
\label{sec:connection:d}
\cite[22.7]{Michor08}
Any connection $\nabla$ induces a one-to-one correspondence between fiber-wise quadratic $C^\al$ mappings $\Ph\colon TN\to TN$ and $C^\al$ sprays $S\colon TN\to TTN$. 
Here $\nabla_{\p_t}c_t = \Ph(c_t)$ corresponds to $c_{tt}=S(c_t)$ for curves $c$ in $N$. 
Equivalently, in terms of the connector $K$, the relation between $\Phi$ and $S$ is as follows:
$$\xymatrix@R=0.5em@C=1em{
& TTN \ar[dl]_{T(\pi_N)} \ar[dr]^{\pi_{TN}} 
\\
TN \ar[dr]_{\pi_N}& & TN \ar[dl]^{\pi_N} 
\\
&N &
}
\hspace{4em}
\xymatrix@R=0.5em@C=1em{
& TTN \ar[dl]_{T(\pi_N)} \ar[dr]^{K} 
\\
TN& & TN  
\\
&TN  \ar@{=}[ul] \ar[ur]_{\Ph} \ar@{-->}[uu]^{S\!}&
}
$$ 
The diagram on the left introduces the projections $T(\pi_N)$ and $\pi_{TN}$, which define the two vector bundle structures on $TTN$.
The diagram on the right shows that $\Phi$ and $S$ are related by $\Phi=K\circ S$. 
\end{enumerate}

The following lemma describes how any connection on $TN$ induces via a product-preserving functor from finite to infinite-dimensional manifolds \cite{KMS93, KM96} a connection on the mapping space $H^s_{H^r}(M,TN)$.
The induced connection will be used as an auxiliary tool for expressing the geodesic equation; see \autoref{thm:geodesic}.   

\begin{lemma}[Induced connection on mapping spaces]
\label{lem:connection}
Let $r \in (\dim(M)/2,\infty)$, $s \in [-r,r]$, and $\al \in \{\infty,\om\}$. 
Then any $C^\al$ connection on $TN$ induces in a natural (i.e., functorial) way a $C^\al$ connection on $H^s_{H^r}(M,TN)$.
\end{lemma}

\begin{proof}
Note that $TN\mapsto H^s_{H^r}(M,TN)$ is a product-preserving functor from finite-dimensional manifolds to infinite-dimensional manifolds as described in \cite{KM96} and \cite[Section 31]{KM97}.
Furthermore, note that $TH^s_{H^r}(M,TN)=H^{r,s,r,s}(M,TTN)$, where $(r,s,r,s)$ denotes the Sobolev regularity of the individual components in any local trivialization $TTN\supset TTU \xrightarrow{TTu} u(U)\x (\mathbb R^n)^3\subset (\mathbb R^n)^4$ induced by a chart 
$N\supset U\xrightarrow{u}u(U)\subset \mathbb R^n$;  cf.~the proof of \autoref{thm:sobolev_mappings}.
Applying the functor $H^s_{H^r}(M,\cdot)$ to the connector $K\colon TTN\to TN$ gives the induced connector 
\begin{equation*}
K_*=H^s_{H^r}(M,K)\colon TH^s_{H^r}(M,TN)\to H^s_{H^r}(M,TN), \quad h\mapsto K\circ h.
\end{equation*}
The induced connector is $C^\al$ by \autoref{lem:pushforward}.
\end{proof}

\section{Sobolev immersions}

This section collects some results about the differential geometry of immersions with Sobolev regularity. 
More specifically, it describes the Sobolev regularity of the induced metric, volume form, normal and tangential projections, and fractional Laplacian, as well as variations of these objects with respect to the immersion.
Here the term fractional Laplacian is understood as a $p$-th power of the operator $1+\Delta$, where $\Delta$ is the Bochner Laplacian and $p \in \mathbb R$; see \cite[Section~3]{bauer2018smooth}.

\begin{lemma}[Geometry of Sobolev immersions]
\label{lem:dependence}
The following statements hold for any $r \in (\operatorname{dim}(M)/2+1,\infty)$ and any smooth Riemannian metric $\g$ on $N$:
\begin{enumerate}
\item
\label{lem:dependence:a}
The space $\Imm_{H^{r}}(M,N)$ of all immersions $f\colon M\to N$ of Sobolev class $H^{r}$ is an open subset of the real analytic manifold $H^{r}(M,N)$.

\item 
\label{lem:dependence:b}
The pull-back metric is well defined and real analytic as a mapping
\begin{equation*}
\Imm^r(M,N) \ni f \mapsto f^*\g \in \Met_{H^{r-1}}(M) 
:=
\Ga_{H^{r-1}}(S^2_+T^*M).
\end{equation*}

\item
\label{lem:dependence:c}
The Riemannian volume form is well defined and real analytic as a mapping
\begin{equation*}
\Imm^r(M,N)  \ni f \mapsto \vol^{f^*\g} \in \Ga_{H^{r-1}}(\Vol M).
\end{equation*}

\item 
\label{lem:dependence:d}
The tangential projection $\top\colon T\Imm(M,N)\to\X(M)$ and the normal projection $\bot\colon T\Imm(M,N)\to T\Imm(M,N)$ are defined for smooth $h\in T_f\Imm(M,N)$ via the relation $h = Tf.h^\top+h^\bot$, where $\g(h^\bot(x),T_xf(T_xM))=0$ for all $x\in M$. 
They extend real analytically for any real number $s\in [1-r,r-1]$ to  
\begin{align*}
&\bot \in \Gamma_{C^{\omega}}\Big(L(H^s_{\Imm^r}(M,TN),H^s_{\Imm^r}(M,TN))\Big),\\
&\top\in C^{\omega}\Big(H^s_{\Imm^r}(M,TN), \X_{H^s}(M)\Big),
\end{align*} 
where $H^s_{\Imm^r}(M,TN)$ is the space of `$H^s$ mappings $M\to TN$ with foot point in $\Imm^r(M,N)$' described in  \autoref{thm:sobolev_mappings}.

\item 
\label{lem:dependence:e}
For any real numbers $s,p$ with $s, s-2p \in [1-r,r]$, the fractional Laplacian
\begin{equation*}
f\mapsto (1+\De^{f^*\g})^p
\end{equation*}
is a real analytic section of the bundle 
\begin{equation*}
 GL(H^s_{\Imm^r}(M,TN),H^{s-2p}_{\Imm^r}(M,TN)).
\end{equation*}
\end{enumerate}
\end{lemma}

\begin{proof}
\begin{enumerate}[wide]
\item The space $H^r(M,N)$ is continuously embedded in $C^1(M,N)$ because $r>\dim(M)/2+1$. Thus, the space of immersions is open.
\item follows from the formula $f^*\g=\g(Tf,Tf)$. 
\item follows from \ref{lem:dependence:b} and the real analyticity of $g \mapsto \vol^g$; see \cite[Lemma~3.3]{bauer2018smooth}. 
\item 
Let $U$ be an open subset of $M$ which carries a local frame $X \in \Ga(GL(\mathbb R^m,TU))$.
For any $f \in \Imm^{r}(M,N)$, the Gram-Schmidt algorithm transforms $X$ into an $(f^*\g)$-ortho\-normal frame $Y_f\in \Ga_{H^{r-1}}(GL(\mathbb R^m,TU))$, which is given by
\begin{align*}
\forall j \in \{1,\dots,m\}:
\qquad
Y_f^j=\frac{X^j-\sum_{k=1}^{j-1} (f^*\g)(Y_f^k,X^j) Y_f^k}{\left\|X^j-\sum_{k=1}^{j-1} (f^*\g)(Y_f^k,X^j) Y_f^j\right\|_{f^*\g}}.
\end{align*}
This defines a real analytic map
\begin{align*}
Y\colon \Imm^r(M,N) \to \Ga_{H^{r-1}}(GL(\mathbb R^m,TU)).
\end{align*}
We write $TN$ as a sub-bundle of a trivial bundle $N\times V$ and denote the corresponding inclusion and projection mappings by 
\begin{equation*}
i\colon TN \to N \times V, 
\qquad
\pi\colon N\times V \to TN.
\end{equation*}
This allows one to define a projection from $N\times V$ onto $TN$ and further onto the normal bundle of $f$, which is real analytic as a map
\begin{align*}
p\colon \Imm^r(M,N) &\to H^{r-1}(U,L(V,V)), 
\\
p_f(x)(v) &:= v-\sum_{i=1}^m \g\big(\pi(f(x),v), T_xf.Y_f^i(x)\big). 
\end{align*}
This construction can be repeated for any open set $\tilde U$ such that $T\tilde U$ is parallelizable, and the resulting projections $p_f$ coincide on $U\cap \tilde U$. 
Thus, one obtains a real analytic map  
\begin{align*}
p\colon \Imm^r(M,N) \to H^{r-1}(M,L(V,V)).
\end{align*}
By the module properties \ref{thm:module}, this induces a real analytic map
\begin{gather*}
\tilde p\colon \Imm^r(M,N) \times H^s(M,V) \to H^s(M,V), 
\quad
\tilde p(f,h) := p_f.h.
\end{gather*}
These maps fit into the commutative diagrams
\begin{equation*}
\xymatrix{
T_{f(x)}N 
\ar[r]^\bot
\ar@{^(->}[d]^i
&
T_{f(x)}N 
\ar@{<<-}[d]^\pi
\\
V
\ar[r]^{p_f(x)}
&
V
}
\hspace{0.5em}
\xymatrix{
H^s_{\Imm^r}(M,TN)
\ar[r]^\bot
\ar@{^(->}[d]^{i_*}
&
H^s_{\Imm^r}(M,TN)
\ar@{<<-}[d]^{\pi_*}
\\
\Imm^r(M,N)\times H^s(M,V)
\ar[r]^{\tilde p}
&
\Imm^r(M,N)\times H^s(M,V)
}
\end{equation*}
The maps $i_*$ and $\pi_*$ are real analytic, as shown in part \ref{lem:pushforward:b'} of the proof of \autoref{lem:pushforward}. 
Therefore, the map $\bot=\pi_*\circ \tilde p \circ i_*$ is real analytic. 
The tangential projection $h^\top = Tf^{-1}(h-h^\bot)$ is then also real analytic.

\item There is a bundle $E$ over $N$ such that $TN\oplus E$ is a trivial bundle, i.e., $TN\oplus E \cong N\times V$ for some vector space $V$.
We endow the bundle $E$ with a smooth connection and the bundle $N\times V \cong TN\oplus E$ with the product connection. 
By construction, the inclusion $i\colon TN\to N\times V$ and projection $\pi\colon N\times V \to TN$ respect the connection. 
At the level of Sobolev sections of these bundles, this means that the natural inclusion and projection mappings fit into the following commutative diagram with $p=1$:
\begin{equation*}
\xymatrix@C6em{
H^s_{\Imm^r}(M,TN)
\ar[r]^{(1+\De)^p}
\ar@{^{(}->}[d]^{i_*}
&
H^{s-2p}_{\Imm^r}(M,TN)
\ar@{<<-}[d]^{\pi_*}
\\
\Imm^r(M,N)\times H^s(M,V)
\ar[r]^{(\Id,(1+\De)^p)}
&
\Imm^r(M,N)\times H^{s-2p}(M,V)
}
\end{equation*}
As the functional calculus preserves commutation relations, this extends to all $p$.
Thus, we have reduced the situation to the bottom row of the diagram, where the fractional Laplacian acts on $H^s(M,V)$.
In this case real analytic dependence of the fractional Laplacian on the metric has been shown in \cite[Theorem 5.4]{bauer2018smooth}.
Now the claim follows from the chain rule and \ref{lem:dependence:b}.
\qedhere
\end{enumerate}
\end{proof}

The following lemma describes the first variation of the metric and fractional Laplacian. 
The key point is that the variation in normal directions is more regular than the variation in tangential directions. 
This will be of importance in \autoref{thm:satisfies}.
The lemma is formulated using an auxiliary connection $\hat\nabla$ on $N$, e.g., the Levi-Civita connection of a Riemannian metric $\g$ on $N$. 

\begin{lemma}[First variation formulas]
\label{lem:variation}
Let $\g$ be a smooth Riemannian metric on $N$, 
and let $\hat \nabla$ be a $C^\al$ connection on $N$ for $\al\in\{\infty,\om\}$.
\begin{enumerate}
\item 
\label{lem:variation:a}
For any $r \in (\operatorname{dim}(M)/2+1,\infty)$ and $s \in [2-r,r]$, the variation of the pull-back metric extends to a real analytic map
\begin{equation*}
H^s_{\Imm^r}(M,TN) \ni m \mapsto D_{f,m}(f^*\g) \in \Ga_{H^{s-1}}(S^2T^*M).
\end{equation*}

\item 
\label{lem:variation:b}
For any $r \in (\operatorname{dim}(M)/2+2,\infty)$ and $s \in [2-r,r-2]$, the variation of the pull-back metric in normal directions extends to a real analytic map
\begin{equation*}
H^s_{\Imm^r}(M,TN) \ni m \mapsto D_{f,m^\bot}(f^*\g) \in \Ga_{H^{s}}(S^2T^*M).
\end{equation*}

\item 
\label{lem:variation:c}
For any $r>\operatorname{dim}(M)/2+2$ and $p \in [1,r-1]$ the variation of the fractional Laplacian in normal directions extends to a $C^\al$ map
\begin{equation*}
H^{2p-r}_{\Imm^r}(M,TN) \ni m \mapsto \hat\nabla_{m^\bot}(1+\Delta^{f^*\bar g})^p 
\in L(H^r_{\Imm^r}(M,TN),H^{1-r}_{\Imm^r}(M,TN)),
\end{equation*}
where $\hat\nabla$ is the induced connection on $GL(H^r_{\Imm^r}(M,TN),H^{1-r}_{\Imm^r}(M,TN))$ described in \autoref{lem:connection}, 
and $L(H^r_{\Imm^r}(M,TN),H^{1-r}_{\Imm^r}(M,N))$ is the vector bundle over $\Imm^r(M,N)$ described in \autoref{thm:sobolev_mappings}.
\end{enumerate}
\end{lemma}

\begin{proof}
We will repeatedly use the module properties \ref{thm:module}. 
\begin{enumerate}[wide]
\item follows from the following formula for the first variation of the pull-back metric \cite[Lemma~5.5]{Bauer2011b}:
\begin{equation*}
D_{f,m}(f^*\g)
=
\g(\nabla m,Tf)+\g(Tf,\nabla m) 
\end{equation*}

\item Splitting the above formula into tangential and normal parts of $m$ yields
\begin{equation*}
D_{f,m}(f^*\g)
=
-2\g(m^\bot,\nabla Tf)+g(\nabla m^\top,\cdot)+g(\cdot,\nabla m^\top).
\end{equation*}
Now the claim follows from the real analyticity of the projection $\bot$ in \autoref{lem:dependence}.

\item[\textbf{(c')}] 
\makeatletter
\protected@edef\@currentlabel{\textbf{(c')}}
\makeatother
\label{lem:variation:c'}
We claim for any bundle $E$ over $M$ with fixed fiber metric and fixed connection (i.e., not depending on $g$) that the following map is real analytic:
\begin{equation*}
\Met_{H^{r-1}}(M) \times \Ga_{H^{s}}(S^2T^*M))\ni (g,m) \mapsto D_{g,m}\Delta^g \in L(\Ga_{H^{q}}(E),\Ga_{H^{s+q-r-1}}(E)),
\end{equation*}
where $s\in [2-r,r-1]$ and $q\in [2-s,r]$.
To prove the claim we proceed similarly to \cite[Lemma~3.8]{bauer2018smooth}. 
As the connection on $E$ does not depend on the metric $g$, 
\begin{align*}
D_{g,m}\De^gh 
&= 
-D_{g,m}(\Tr^{g^{-1}}\nabla^g\nabla h) 
= 
-(D_{g,m}\Tr^{g^{-1}})\nabla^g\nabla h 
-\Tr^{g^{-1}}(D_{g,m}\nabla^g)\nabla h.
\end{align*}
Here $\nabla^g$ is the covariant derivative on $T^*M\otimes E$. 
The proof of \cite[Lemma~3.8]{bauer2018smooth} and some multi-linear algebra show that $D_{g,m}\nabla^g$ is tensorial and real analytic as a map
\begin{multline*}
\Met_{H^{r-1}}(M)\times \Ga_{H^{s}}(S^2T^*M) \ni (g,m) 
\\
\mapsto D_{g,m}\nabla^g \in \Ga_{H^{s-1}}(T^*M\otimes L(T^*M\otimes E,T^*M\otimes E)).
\end{multline*}
Moreover, the following maps are real analytic by  \cite[Lemmas~3.2 and~3.5]{bauer2018smooth}:
\begin{align*}
\Met_{H^{r-1}}(M) \ni g &\mapsto g\i \in \Ga_{H^{r-1}}(S^2TM), 
\\
\Met_{H^{r-1}}(M) \ni g &\mapsto \nabla^g \in  L(\Ga_{H^{q-1}}(T^*M\otimes E),\Ga_{H^{q-2}}(T^*M\otimes T^*M\otimes E)).
\end{align*}
Together with the module properties~\ref{thm:module} this establishes  \ref{lem:variation:c'}.

\item [\textbf{(c'')}]
\makeatletter
\protected@edef\@currentlabel{\textbf{(c'')}}
\makeatother
\label{lem:variation:c''} 
Using \ref{lem:variation:c'} we will now study the smooth dependence of  fractional Laplacians. In particular we claim for any bundle $E$ over $M$ with fixed fiber metric and fixed connection and any $p \in (1,r-1]$ that the following map is real analytic:
\begin{equation*}
\Met_{H^{r-1}}(M) \times \Ga_{H^{2p-r}}(S^2T^*M))\ni (g,m) \mapsto D_{g,m}(1+\Delta^g)^p \in L(\Ga_{H^{r}}(E),\Ga_{H^{1-r}}(E)).
\end{equation*}
The claim is a generalization of \cite[Lemma~5.5]{bauer2018smooth} to perturbations $m$ with even lower Sobolev regularity and uses the fact that the connection on $E$ does not depend on the metric $g$.
Let $X,Y,Z$ be the spaces of operators given by
\begin{align*}
X&=L(\Ga_{H^{r}}(E),\Ga_{H^{r-2}}(E))\cap L(\Ga_{H^{3-r}}(E),\Ga_{H^{1-r}}(E)),
\\
Y&=L(\Ga_{H^{r}}(E),\Ga_{H^{-r+2p-1)}}(E))\cap L(\Ga_{H^{r-2p+2}}(E),\Ga_{H^{1-r}}(E)),
\\
Z&=L(\Ga_{H^{r}}(E),\Ga_{H^{r-2}}(E))\cap L(\Ga_{H^{r-2p+2}}(E),\Ga_{H^{r-2p}}(E)).
\end{align*}
Note that the conditions $r>2$ and $p>1$ ensure that $X$, $Y$, and $Z$ are intersections of operator spaces on distinct Sobolev scales, as required in \cite[Theorem 4.5]{bauer2018smooth}
Moreover, let $U\subseteq X$ be an open neighborhood of $1+\De^g$ with $g \in \Met_{H^{r-1}}(M)$ such that the holomorphic functional calculus is well-defined and holomorphic on $U$ in the sense of \cite[Theorem 4.5]{bauer2018smooth}.
Then the desired map is the composition of the following two maps: 
\begin{gather*}
\Met_{H^{r-1}}(M)\times \Ga_{H^{2p-r}}(S^2T^*M) \in (g,m)\mapsto (1+\De^g,D_{g,m}\De^g) \in (X,Y),
\\
(U,Y) \ni (A,B) \mapsto D_{A,B}A^p \in L(\Ga_{H^{r}}(E),\Ga_{H^{1-r}}(E)).
\end{gather*}
The first map is real analytic by \autoref{lem:dependence}.\ref{lem:dependence:e} and \ref{lem:variation:c'}.
The second map has to be interpreted via the following identity, which is shown in the proof of \cite[Lemma~5.5]{bauer2018smooth} using the resolvent representation of the functional calculus:
\begin{equation*}
\forall A \in U, \forall B \in Y\cap Z: 
\qquad 
D_{A,B}A^p
=
A^{r-1-p}D_{A,A^{p-r+1}B}A^p.
\end{equation*}
The right-hand side above is the composition of the following maps, which are again real analytic by \cite[Theorem 4.5]{bauer2018smooth}:
\begin{gather*}
(U,Y) \ni (A,B) \mapsto (A,A^{p-r+1/2}B) \in (U,Z),
\\
(U,Z) \ni (A,B) \mapsto (A,D_{A,B}A^p) \in U\times L(\Ga_{H^{r}}(E),\Ga_{H^{r-2p}}(E))
\\
U\times L(\Ga_{H^{r}}(E),\Ga_{H^{r-2p}}(E)) \ni (A,B) \mapsto A^{r-p-1/2}B \in L(\Ga_{H^{r}}(E),\Ga_{H^{1-r}}(E))
\end{gather*}
This proves \ref{lem:variation:c''}. 
Note that \ref{lem:variation:c''} extends to $p=1$ thanks to \ref{lem:variation:c'}.

\item
As in the proof of \autoref{lem:dependence}.\ref{lem:dependence:e}, we write $i$ and $\pi$ for the inclusion and projection mappings of $TN$, seen as a sub-bundle of a trivial bundle $TN\oplus E\cong N\times V$ with $C^\al$ product connection.
If we consider $i_*$ and $\pi_*$ as real analytic sections of operator bundles,
\begin{align*}
i_* \in \Gamma_{C^\om}(L(H^r_{\Imm^r}(M,TN),H^r_{\Imm^r}(M,N\times V)), 
\\
\pi_* \in \Gamma_{C^\om}(L(H^{r-2p}_{\Imm^r}(M,TN),H^{r-2p}_{\Imm^r}(M,N\times V)),
\end{align*}
then the covariant derivative of the fractional Laplacian can be expressed as follows:
\begin{multline*}
\hat\nabla_{m^\bot}(1+\Delta^{f^*\g})^p
=
(\hat\nabla_{m^\bot}\pi_*)(\Id,(1+\Delta^{f^*\g})^p)i_*
\\
+
\pi_*\big(\hat\nabla_{m^\bot}(\Id,(1+\Delta^{f^*\g})^p)\big)i_*
+
\pi_*(\Id,(1+\Delta^{f^*\g})^p)(\hat\nabla_{m^\bot}i_*).
\end{multline*}
The maps $i_*$ and $\pi_*$ are real analytic, and consequently their covariant derivatives are $C^\al$. 
According to \autoref{lem:connection}, the canonical connection $D$ on the vector space $V$ induces a real analytic connection on the bundle of bounded linear operators $L(H^r_{\Imm^r}(M,N\times V), H^{1-r}_{\Imm^r}(M,N\times V))$.
By general principles, this connection differs from $\hat\nabla$ by a $C^\al$ tensor field, often called the Christoffel symbol. 
Thus, it suffices to show that the following map is $C^\al$: 
\begin{multline*}
H^{2p-r}_{\Imm^r}(M,TN) \ni m \mapsto D_{f,m^\bot}(\Id,(1+\Delta^{f^*\g})^p) 
\\
\in L(H^r_{\Imm^r}(M,N\times V),H^{1-r}_{\Imm^r}(M,N\times V)).
\end{multline*}
As $D$ is the canonical connection, this is equivalent to the following map being $C^\al$:
\begin{equation*}
H^{2p-r}_{\Imm^r}(M,TN) \ni m \mapsto D_{f,m^\bot} (1+\Delta^{f^*\g})^p \in L(H^r(M,V),H^{1-r}(M,V)).
\end{equation*}
By \ref{lem:variation:b} with $s=2p-r$, the variation of the pull-back metric in normal directions is real analytic as a map
\begin{align*}
H^{2p-r}_{\Imm^r}(M,TN) \ni m \mapsto D_{f,m^\bot}(f^*\g) \in \Ga_{H^{2p-r}}(S^2T^*M).
\end{align*}
Thus, \ref{lem:variation:c} follows from \ref{lem:variation:c''} and the chain rule.
\qedhere
\end{enumerate}
\end{proof}

\section{Weak Riemannian metrics on spaces of immersions}
\label{sec:metrics} 

The main result of this section is that the geodesic equation of Sobolev-type metrics is locally well posed under certain conditions on the operator governing the metric. 
The setting is general and encompasses  several examples, including in particular fractional Laplace operators.

\subsection{Sobolev-type metrics}
\label{sec:conditions}
Within the setup of \autoref{sec:setting}, we consider Sobolev-type Riemannian metrics on the space of immersions $f\colon M\to N$ of the form
\begin{align*}
G^P_f(h,k)&=\int_M \g(P_fh,k) \vol(f^*\g),
\qquad
h,k \in T_f\Imm(M,N),
\end{align*}
where $\g$ is a $C^\al$ Riemannian metric on $N$ for $\al\in\{\infty,\om\}$, 
and where $P$ is an operator field which satisfies the following conditions for 
some $p \in [0,\infty)$, 
some $r_0 \in (\dim(M)/2+1,\infty)$, 
and all $r \in [r_0,\infty)$:
\begin{enumerate}
\item
\label{sec:conditions:a}
Assume that $P$ is a $C^\al$ section of the bundle
\begin{equation*}
GL(H^r_{\Imm^r}(M,TN),H^{r-2p}_{\Imm^r}(M,TN)) \to \Imm^r(M,N),	
\end{equation*}
where $GL$ denotes bounded linear operators with bounded inverse. 

\item 
\label{sec:conditions:b}
Assume that $P$ is $\Diff(M)$-equivariant in the sense that one has for all $\ph\in\Diff(M)$, $f \in \Imm^r(M,N)$, and $h\in T_f\Imm^r(M,N)$ that
\begin{equation*}
(P_fh)\circ\ph = P_{f\circ\ph}(h\circ\ph).
\end{equation*}

\item 
\label{sec:conditions:c}
Assume for each $f\in \Imm^r(M,N)$ that the operator $P_f$ is nonnegative and symmetric with respect to the $H^0(g)$ inner product on $T_f\Imm^r(M,N)$, i.e., for all $h,k \in T_f\Imm^r(M,N)$:
\begin{equation*}
\int_M \g(P_f h,k)\vol(g) = \int_M \g(h,P_f k)\vol(g),
\qquad
\int_M \g(P_f h,h)\vol(g) \geq 0.
\end{equation*}

\item
\label{sec:conditions:d}
Assume that the normal part of the adjoint $\Adj(\nabla P)^\bot$, defined by 
\begin{equation*}
\int_M \g((\nabla_{m^\bot}P)h,k) \vol(g) 
=
\int_M \g(m,\Adj(\nabla P)^\bot(h,k)) \vol(g)
\end{equation*}
for all $f \in \Imm(M,N)$ and $m,h,k \in T_f\Imm$, exists and is a $C^\al$ section of the bundle of bilinear maps
\begin{equation*}
L^2(H^r_{\Imm^r}(M,TN),H^r_{\Imm^r}(M,TN);H^{r-2p}_{\Imm^r}(M,TN)).
\end{equation*} 
Here $\nabla$ denotes the induced connection (see \autoref{lem:connection}) of the Levi-Civita connection of $\g$.
\end{enumerate}

\begin{remark}
In \cite[Section~6.6]{Bauer2011b} we had more complicated conditions, and we implicitly claimed that they imply the conditions in \autoref{sec:conditions} above.
There was, however, a significant gap in the argumentation of the main result.
Namely, we did not show the smoothness of the extended mappings on Sobolev completions.
This article closes this gap and extends the analysis to the larger class of fractional order metrics.
\end{remark}

We now derive the geodesic equation of Sobolev-type metrics. 
Recall that the usual form of the geodesic equation is $f_{tt}=\Gamma_f(f_t,f_t)$, where the time derivatives $f_t$ and $f_{tt}$ as well as the Christoffel symbols $\Gamma$ are expressed in a chart. 
This raises the problem that the space $\Imm(M,N)$ lacks canonical charts, unless $N$ admits a global chart.
However, $\Imm(M,N)$ carries a canonical connection, namely, the one induced by the metric $\g$ on $N$, which has been described in \autoref{lem:connection}.
This auxiliary connection, which will be denoted by $\nabla$, allows one to write the geodesic equation as $\nabla_{\partial_t}f_t=\Gamma_f(f_t,f_t)$, where $\Gamma$ is a difference between two connections and therefore tensorial.
In the special case where $N$ is an open subset of Euclidean space, this coincides with the usual derivative $\nabla_{\partial_t}f_t=f_{tt}$; cf.~\autoref{cor:flat}.

\begin{theorem}[Geodesic equation]
\label{thm:geodesic}
\cite[Theorem~6.4]{Bauer2011b}
Assume the conditions of \autoref{sec:conditions}.
Then a smooth curve $f\colon [0,1]\to \Imm(M,N)$ is a critical point of the energy functional 
\begin{equation*}
E(f)
=
\frac12 \int_0^1 \int_M \g(P_f f_t, f_t)\vol^g dt
\end{equation*}
if and only if it satisfies the geodesic equation 
\begin{align*}
\nabla_{\p_t} f_t= &\frac12P_f\i\Big(\Adj(\nabla P)_f(f_t,f_t)^\bot-2Tf\,\g(P_ff_t,\nabla f_t)^\sharp
-\g(P_f f_t,f_t)\,\Tr^g(\nabla Tf)\Big)
\\&
-P_f\i\Big((\nabla_{f_t}P)f_t+\Tr^g\big(\g(\nabla f_t,Tf)\big) P_ff_t\Big).
\end{align*}
This also holds for smooth curves in $\Imm^r(M,N)$ for any $r\geq r_0$.
\end{theorem}

Here we used the following notation: 
$g=f^*\g$ is the pull-back metric and $\sharp=g^{-1}$ its associated musical isomorphism, 
the operator $P$ is seen as a map $P\colon \Imm \to GL( T\Imm,T\Imm)$, 
its composition with $f$ is denoted by $P_f\colon\R\to  GL( T\Imm,T\Imm)$,
its covariant derivative with respect to the connection on $L(T\Imm,T\Imm)$ induced by $\nabla$ is denoted by $\nabla P\colon T\Imm\to L(T\Imm,T\Imm)$,
the canonical vector field on $\R$ is denoted by $\p_t\colon \mathbb R \to T\mathbb R$,
the time derivative $f_t=\p_t f$ is viewed as a map $f_t\colon \R\times M\to TN$ in the expression $\nabla f_t\colon \R\times M \to T^*M\otimes f^*TN$ and as a map $f_t\colon \mathbb R\to  T\Imm$ elsewhere, 
the spatial derivative $Tf$ is viewed as a map $Tf\colon \R\times M \to T^*M\otimes f^*TN$,
and the map $\nabla Tf\colon \R\times M \to T^*M\otimes T^*M\otimes f^*TN$ is the second fundamental form.

\begin{proof}
We will consider variations of the curve energy functional
along one-parameter families $f\colon (-\ep,\ep) \x [0,1] \x M \to N$ of curves of immersions with fixed endpoints. 
The variational parameter will be denoted by $s \in (-\ep,\ep)$, 
the time-parameter by $t \in [0,1]$.
Then the first variation of the energy $E(f)$ can be calculated as follows:
\begin{equation*}
\p_s E(f)
=
\p_s \frac12 \int_0^1 \int_M \g(P_f f_t, f_t)\vol^g dt.
\end{equation*}
As the connection respects $\g$ and is a derivation of tensor products, and as the operator $P_f$ is symmetric, we have
\begin{align*}
\p_s E(f)
=
\frac12\int_0^1\int_M  \g\left(
(\nabla_{\p_s} P_f)f_t+2P_f\nabla_{\p_s} f_t+ \frac{\partial_s \vol^g}{\vol^g} P_f f_t, f_t \right) \vol^gdt.
\end{align*}
We will treat each of the three summands above  separately, making extensive use of properties \ref{sec:connection} of the (induced) connection $\nabla$:
\begin{enumerate}[wide]
\item
\label{thm:geodesic:a}
For the first summand we have by the definition of the adjoint that
\begin{multline*}
\frac12\int_0^1\int_M  \g(
(\nabla_{\p_s} P_f)f_t,f_t) \vol^gdt
=
\frac12\int_0^1\int_M  \g(
(\nabla_{f_s} P)f_t,f_t) \vol^gdt
\\=
\frac12\int_0^1\int_M  \g\Big(
f_s,\Adj(\nabla P)(f_t,f_t)^{\bot}+\Adj(\nabla P)(f_t,f_t)^{\top}\Big)
 \vol^gdt.
\end{multline*}
To calculate the tangential part of the adjoint, thereby establishing its existence, we need the following formula for the tangential variation of $P$, which holds for any vector field $X$ on $M$:
\begin{align*}
(\nabla_{Tf.X} P)(h) &=
(\nabla_{\p_t|_0} P_{f\circ \Fl_t^X})(h \circ \Fl_0^X) \\&=
\nabla_{\p_t|_0}\big(P_{f\circ \Fl_t^X}(h \circ \Fl_t^X)\big) -
P_{f\circ \Fl_0^X} \big( \nabla_{\p_t|_0}(h \circ \Fl_t^X)\big) \\&=
\nabla_{\p_t|_0}\big(P_f(h) \circ \Fl_t^X\big) -
P_f \big( \nabla_{\p_t|_0}(h \circ \Fl_t^X)\big) \\&=
\nabla_X\big(P_f(h)) - P_f \big( \nabla_X h\big),
\end{align*}
where $Fl_t^X$ denotes the flow of the vector field $X$ at time $t$ and where we used the equivariance of $P$ in the step 
from the second to the third line. 
Using this and the symmetry of $P$ we get
\begin{align*} &
\frac12\int_0^1\int_M  \g\Big(f_s,\Adj(\nabla P)(f_t,f_t)^{\top}\Big) \vol^gdt=
\int_M \g\big((\nabla_{Tf.f_s^\top} P)f_t,f_t\big) \vol(g)  \\&\qquad=
\int_M \g\big(\nabla_{f_s^\top}(P_ff_t)-P_f(\nabla_{f_s^\top}f_t),f_t\big) \vol(g) \\&\qquad=
\int_M \big(\g(\nabla_{f_s^\top}(P_ff_t),f_t)-\g(\nabla_{f_s^\top}f_t,P_ff_t)\big) \vol(g) \\&\qquad=
\int_M \g\Big(Tf .f_s^\top,Tf.\big(\g(\nabla (P_ff_t),f_t)-\g(\nabla f_t,P_ff_t)\big)^\sharp\Big) \vol(g)\\&\qquad=
\int_M \g\Big(Tf .f_s^\top,Tf.\big(\nabla \g( P_f f_t,f_t)-2\g(\nabla f_t,P_f f_t)\big)^\sharp\Big) \vol(g)\\&\qquad=
\int_M \g\Big(f_s,Tf.\big(\nabla \g( P_f f_t,f_t)-2\g(\nabla f_t,P_f f_t)\big)^\sharp\Big) \vol(g).
\end{align*}
Thus we obtain the following formula for the first summand of the variation of $E$:
\begin{multline*}
\frac12\int_0^1\int_M  \g(
(\nabla_{\p_s} P_f)f_t,f_t) \vol^gdt\\=
\frac12\int_0^1\int_M  \g\Big(
f_s,\Adj(\nabla P)(f_t,f_t)^{\bot}+Tf.\big(\nabla\g( Pf_t,f_t)-2\g(\nabla f_t,Pf_t)\big)^\sharp\Big)
 \vol^gdt.
\end{multline*}

\item
\label{thm:geodesic:b}
As $P_f$ is symmetric and  the covariant derivative on $\Imm(M,N)$ is torsion-free (see \autoref{sec:connection}), i.e.,
\begin{equation*}
\nabla_{\p_t}f_s - \nabla_{\p_s}f_t=Tf.[\p_t,\p_s]+\on{Tor}(f_t,f_s) = 0,
\end{equation*}
we get for the second summand
\begin{equation*}
\int_0^1\int_M  \g\left(
P_f\nabla_{\p_s} f_t, f_t \right) \vol^gdt= \int_0^1\int_M  \g\left(
\nabla_{\p_t} f_s, P_f f_t \right) \vol^gdt.
\end{equation*}
Integration by parts for $\p_t$ 
yields
\begin{multline*}
\int_0^1\int_M  \g\left(
\nabla_{\p_t} f_s, P_f f_t \right) \vol^gdt\\= \int_0^1\int_M \bigg( \g\left(
 f_s, - (\nabla_{f_t}P)f_t-P_f(\nabla_{f_t}f_t) \right)- \frac{{\p_t} \vol^g}{\vol^g} \bigg)\vol^gdt
\end{multline*}
To further expand the last term we use the following formula for the variation of the volume form \cite[Lemma~5.7]{Bauer2011b}:
\begin{equation*}
\frac{\p_t \vol^g}{\vol^g} = 
\Tr^g\big(\g(\nabla f_t,Tf)\big) 
=
-\nabla^*( \g(f_t,Tf))-\g\big(f_t,\Tr^g(\nabla Tf)\big),
\end{equation*}
where $\nabla Tf$ is the second fundamental form and where $\nabla^*$ denotes the adjoint of the covariant derivative.
Using the first of the above formulas we obtain for the second summand:
\begin{multline*}
\int_0^1\int_M  \g\left(
\nabla_{\p_t} f_s, P_f f_t \right) \vol^gdt\\= \int_0^1\int_M \bigg( \g\left(
 f_s, - (\nabla_{f_t}P)f_t-P_f(\nabla_{f_t}f_t) \right)-\Tr^g\big(\g(\nabla f_t,Tf)\big) P_f f_t \bigg)\vol^gdt.
\end{multline*}

\item
\label{thm:geodesic:c}
Using the second version of the variational formula for the volume in the third summand in the variation of the energy yields
\begin{align*}
&\frac12\int_0^1\int_M 
 \frac{\partial_s \vol^g}{\vol^g}  \g\left(P_f f_t, f_t \right) \vol^gdt\\
&\qquad =- \frac12\int_0^1\int_M 
\Big(\nabla^*( \g(f_s,Tf))+\g\big(f_s,\Tr^g(\nabla Tf)\big) \Big) \g\left(P_f f_t, f_t \right) \vol^gdt\\
&\qquad=-\frac12\int_0^1\int_M \bigg(\g(f_s, Tf.(\nabla \g\left(P_f f_t, f_t \right))^{\sharp}+\Tr^g(\nabla Tf) \g\left(P_f f_t, f_t \right)\bigg) \vol^gdt.
\end{align*}
\end{enumerate}
Taken together, the calculations of \ref{thm:geodesic:a}--\ref{thm:geodesic:c} yield
\begin{multline*}
\p_s E(f)
=
\frac12\int_0^1\int_M  \g\bigg(f_s,\Adj(\nabla P)(f_t,f_t)^{\bot}-2Tf.\g(\nabla f_t,Pf_t)^\sharp-2(\nabla_{f_t}P)f_t
\\-2P_f(\nabla_{f_t}f_t)-2\Tr^g\big(\g(\nabla f_t,Tf)\big) P_f f_t -\Tr^g(\nabla Tf) \g\left(P_f f_t, f_t \right)  \bigg) \vol^gdt.
\end{multline*}
Setting $\partial_s E(f)=0$ for arbitrary perturbations $f_s$ yields the geodesic equation on the space $\Imm(M,N)$ of smooth immersions.
This statement extends to the space $\Imm^r(M,N)$ of Sobolev immersions because the right-hand side of the geodesic equation is continuous in $f \in C^\infty([0,1],\Imm^r(M,N))$, as shown in part \ref{thm:wellposed:a} of the proof of \autoref{thm:wellposed}.
\qedhere
\end{proof}

We next show well-posedness of the geodesic equation using the Ebin--Marsden approach \cite{EM1970} of extending the geodesic spray to a smooth vector field on  $T\Imm^r$ for sufficiently high $r$ and showing that the solutions exist on a time interval which is independent of $r$. 

\begin{theorem}[Local well-posedness of the geodesic equation]
\label{thm:wellposed}
Assume the conditions of \autoref{sec:conditions} with $p\geq 1$.
Then the following statements hold for all $r \in [r_0,\infty)$:
\begin{enumerate}
\item 
\label{thm:wellposed:a}
The initial value problem for geodesics has unique local solutions in $\Imm^r(M,N)$. 
The solutions depend $C^\al$ on $t$ and on the initial condition $f_t(0)\in T\Imm^r(M,N)$.

\item
\label{thm:wellposed:b}
The Riemannian exponential map $\exp^{P}$ exists and is $C^\al$ on a neighborhood of the zero section in $T\Imm_{H^r}$, 
and $(\pi,\exp^{P})$ is a diffeomorphism from a (smaller) neighborhood of the zero section to a neighborhood of the diagonal in $\Imm^r(M,N)\x \Imm^r(M,N)$. 

\item
\label{thm:wellposed:c}
The neighborhoods in \ref{thm:wellposed:a}--\ref{thm:wellposed:b} are uniform in $r$ and can be chosen open in the $H^{r_0}$ topology.
Thus, \ref{thm:wellposed:a}--\ref{thm:wellposed:b} continue to hold for $r=\infty$, i.e., on the Fr\'echet manifold $\Imm(M,N)$ of smooth immersions.
\end{enumerate}
\end{theorem}

\begin{proof} 
\begin{enumerate}[wide]
\item
This can be shown as in \cite[Theorem~6.6]{Bauer2011b}. 
Let $\Ph(f_t)$ denote the right-hand side of the geodesic equation, i.e.,  
\begin{multline*}
\Ph(f_t)
=
\frac12P\i\Big(\Adj(\nabla P)(f_t,f_t)^\bot-2Tf\,\g(Pf_t,\nabla f_t)^\sharp
-\g(Pf_t,f_t)\,\Tr^g(\nabla Tf)\Big)
\\
-P\i\Big((\nabla_{f_t} P)f_t+\Tr^g\big(\g(\nabla f_t,Tf)\big) Pf_t\Big).
\end{multline*}
A term-by-term investigation using the conditions \ref{sec:conditions} and the module properties~\ref{thm:module} shows that $\Phi$ is a fiber-wise quadratic $C^\al$ map
\begin{align*}
\Phi\colon T\Imm^r(M,N)\to T\Imm^r(M,N).
\end{align*}
Here the condition $p\geq 1$ is needed to ensure that the term $P^{-1}\big(\g(Ph,h)\,\Tr^g(\nabla Tf)\big)$ is again of regularity $H^r$.
The map $\Phi$ corresponds uniquely to a $C^\al$ spray $S$ via the induced connection described in \autoref{lem:connection}.
In more detail: 
The right-hand side diagram in the proof of \autoref{sec:connection}.\ref{sec:connection:d} holds for any manifold $N$ with connector $K$. 
Thus, replacing $(N,K)$ by $(\Imm^r(M,N),K_*)$, one obtains the diagram 
$$
\xymatrix@R=2mm{
&TT\Imm^r (M,N) \ar[dl]_{T(\pi_N)_*} \ar[dr]^{K_*} 
\\
T\Imm^r(M,N)& & T\Imm^r(M,N) 
\\
&T\Imm^r(M,N)  \ar@{=}[ul] \ar[ur]_{\Ph} \ar@{-->}[uu]^{S}&
}
$$
The spray $S$ is $C^\al$ because the connection $K_*$ and the map $\Ph$ are $C^\al$.
Therefore, by the theorem of Picard-Lindel\"of, $S$ admits a $C^\al$ flow  
\begin{equation*}
\on{Fl}^S\colon U \to T\Imm^{r}(M,N)
\end{equation*}
for a maximal  open neighborhood $U$ of $\{0\}\x T\Imm^r(M,N)$ in $\mathbb R\x T\Imm^r(M,N)$. 
The neighborhood $U$ is $\Diff(M)$-invariant thanks to the $\Diff(M)$-equivariance of $S$.

\item follows from \ref{thm:wellposed:a} as in \cite[Theorem~6.6]{Bauer2011b}, and \ref{thm:wellposed:c} follows from \autoref{lem:nolossnogain} by writing $\Imm(M,N)$ as the intersection of all $\Imm^{r_0+k}(M,N)$ with $k \in \mathbb N_{\geq 0}$.
\qedhere
\end{enumerate}
\end{proof}

\begin{corollary}
\autoref{thm:wellposed} with $\al=\om$ remains valid if the assumptions in \autoref{sec:conditions} are modified as follows: 
the metric $\g$ is only $C^\infty$, 
and the connection $\nabla$ in condition~\ref{sec:conditions:d} is replaced by an auxiliary connection $\hat\nabla$, which is induced by a torsion-free $C^\om$ connection on $N$, as described in \autoref{lem:connection}.
\end{corollary}

\begin{proof}
In the proof of \autoref{thm:geodesic}, the geodesic equation is derived by expressing the first variation $\p_s E$ of the energy functional using the Levi-Civita connection of $\g$. 
If the auxiliary connection $\hat\nabla$ is used  instead, then the following additional terms appear in the formula for $\p_s E$:
\begin{multline*}
\int_0^1\int_M \bigg( -\frac12
(\hat\nabla_{Tf.f_s^\top}\g)(P_ff_t,f_t)
-(\hat\nabla_{f_t}\g)(f_s,P_ff_t)
+\frac12(\hat\nabla_{f_s}\g)(P_ff_t,f_t)\bigg)\vol^gdt
\end{multline*} 
Accordingly, letting $\Psi$ denote the right-hand side of the original geodesic equation with $\nabla P$ replaced by $\hat\nabla P$, i.e., 
\begin{multline*}
\Psi(f_t)
=
\frac12P\i\Big(\Adj(\hat\nabla P)(f_t,f_t)^\bot-2Tf\,\g(Pf_t,\nabla f_t)^\sharp
-\g(Pf_t,f_t)\,\Tr^g(\nabla Tf)\Big)
\\
-P\i\Big((\hat\nabla_{f_t} P)f_t+\Tr^g\big(\g(\nabla f_t,Tf)\big) Pf_t\Big),
\end{multline*}
the geodesic equation becomes 
\begin{multline*}
\hat\nabla_{\p_t}f_t
=
\Psi(f_t) 
-\frac12 Tf.\big(\g^{-1}(\hat\nabla\g)(P_ff_t,f_t)\big)^\top
-\g^{-1}(\hat\nabla_{f_t}\g)(\cdot,P_ff_t) 
\\ 
+\frac12 \g^{-1}(\hat\nabla\g)(P_ff_t,f_t).
\end{multline*}
One verifies as in the proof of \autoref{thm:wellposed} that the right-hand side, seen as a function of $f_t$, is a fiber-wise quadratic real analytic map $T\Imm^r(M,N)\to T\Imm^r(M,N)$.
As the auxiliary connection $\hat\nabla$ is real analytic, this implies that the corresponding spray is real analytic, as well; see \autoref{sec:connection}. 
Since the spray is independent of the auxiliary connection $\hat\nabla$, one may proceed as in the proof of \autoref{thm:wellposed}.
\end{proof}

The following theorem shows that (scale-invariant) fractional-order Sobolev metrics satisfy the conditions in \autoref{sec:conditions}.
This implies local well-posedness of their geodesic equations by \autoref{thm:wellposed}.
Further metrics considered in the literature include curvature weighted metrics and the so-called general elastic metric \cite{jermyn2017}, which can also be formulated in the present framework \cite{bauer2012sobolev}.
The proof takes advantage of the fact that the adjoint in the geodesic equation \ref{thm:geodesic} has been split into normal  and tangential parts.
The normal part has the correct Sobolev regularity thanks to \autoref{lem:variation}.
The tangential part incurs a loss of derivatives, but the bad terms cancel out with some other terms in the geodesic equation as shown in part \ref{thm:geodesic:a} of the proof of  \autoref{thm:geodesic}.

\begin{theorem}
\label{thm:satisfies}
The following operators satisfy the conditions in \autoref{sec:conditions} with $\al=\om$ for any $p \in [1,\infty)$ and $r_0 \in (\operatorname{dim}(M)/2+2,\infty)\cap[p+1,\infty)$:
\begin{align*}
P_f:=\big(1+\De^{f^*\g}\big)^p, 
\qquad
\text{and}
\qquad
P_f:=\Big(\Vol^{-1-\tfrac{2}{\on{dim}M}}+\Vol^{-1}\De^{f^*\g}\Big)^p.
\end{align*}
Thus, the geodesic equations of these metrics are well posed in the sense of \autoref{thm:wellposed}.
\end{theorem}

\begin{proof}
We will prove this result only for the first field of operators because the proof for the second one is analogous.
We shall check conditions \ref{sec:conditions:a}--\ref{sec:conditions:d} of \autoref{sec:conditions}. 
\begin{enumerate}[wide]
\item follows from \autoref{lem:dependence}.

\item 
$\Diff(M)$-equivariance of $(1+\De^{f^*\g})$ is well-known for smooth $f$ and follows in the general case by approximation, noting that the pull-back along a smooth diffeomorphism is a bounded linear map between Sobolev spaces of the same order of regularity \cite[Theorem~B.2]{inci2013regularity}.
As the functional calculus preserves commutation relations, this implies the $\Diff(M)$-equivariance of $(1+\De^g)^p$.

\item is well-known for smooth $f,h,k$ and follows in the general case by approximation using the continuity of $f\mapsto\langle\cdot,\cdot\rangle_{H^0(f^*\g)}$ established in \cite[Lemma 3.3]{bauer2018smooth} and the continuity of $f\mapsto P_f$.

\item 
Recall from \autoref{lem:variation} that the derivative of $P_f$ in normal directions extends to a real analytic map
\begin{multline*}
H^{2p-r}_{\Imm^r}(M,TN) \ni m \mapsto \big(h\mapsto\hat\nabla_{m^\bot}P_fh\big) 
\in L(H^r_{\Imm^r}(M,TN),H^{1-r}_{\Imm^r}(M,TN)).
\end{multline*}
Equivalently, the following map is real analytic:
\begin{multline*}
H^{r}_{\Imm^r}(M,TN) = T\Imm^r(M,N) \ni h \mapsto \big(m\mapsto\hat\nabla_{m^\bot} P_f h\big) 
\\
\in L(H^{2p-r}_{\Imm^r}(M,TN),H^{1-r}_{\Imm^r}(M,N)).
\end{multline*}
Dualization using the $H^0(g)$ duality shows that the adjoint is real analytic
\begin{equation*}
T\Imm^r(M,N) \ni h \mapsto \Adj(\hat\nabla P)(h,\cdot)^\bot \in L(H^{r-1}_{\Imm^r}(M,TN),H^{r-2p}_{\Imm^r}(M,TN)).
\end{equation*}
In particular, the adjoint is real analytic
\begin{equation*}
T\Imm^r(M,N) \ni h \mapsto \Adj(\hat\nabla P)(h,\cdot)^\bot \in L(H^r_{\Imm^r}(M,TN),H^{r-2p}_{\Imm^r}(M,TN)).
\qedhere
\end{equation*}
\end{enumerate}
\end{proof}

\begin{remark}
For Sobolev metrics of integer order $p \in \mathbb N_{>0}$, condition~\ref{sec:conditions:d} of \autoref{sec:conditions} can be verified directly by a term-by-term investigation of the following explicit formula for the normal part of the adjoint \cite[Section~8.2]{Bauer2011b}, assuming that $\hat\nabla=\nabla$ is the Levi-Civita connection of $\g$:
\begin{align*}
\Adj(\nabla P)(h,k)^\bot&=
2\sum_{i=0}^{p-1}\Tr\big(g\i \na Tf g\i \g(\nabla(1+\Delta)^{p-i-1}h,\nabla(1+\Delta)^{i}k ) \big)
\\&\qquad
+\sum_{i=0}^{p-1} \big(\nabla^*\g(\nabla(1+\Delta)^{p-i-1}h,(1+\Delta)^{i}k) \big) \Tr^g(\na Tf)\\
&\qquad+\sum_{i=0}^{p-1}\Tr^g\big(R^{\g}((1+\Delta)^{p-i-1}h,\nabla(1+\Delta)^{i}k)Tf \big)\\
&\qquad-\sum_{i=0}^{p-1}\Tr^g\big(R^{\g}(\nabla(1+\Delta)^{p-i-1}h,(1+\Delta)^{i}k)Tf \big).
\end{align*}
Here $g=f^*\g$, $\Delta=\Delta^{g}$, $\nabla=\nabla^g$, and $R^{\g}$ denotes the curvature on $(N,\g)$. 
This direct calculation is consistent with the more general argument of \autoref{thm:satisfies}.
\end{remark}

\section{Special cases}
\label{sec:special}

This section describes several applications of the general well-posedness result, \autoref{thm:wellposed}. 
First, we consider the geodesic equation of right-invariant Sobolev metrics on the diffeomorphism group $\Diff(M)$.
In Eulerian coordinates, this equation is called Euler--Arnold \cite{Ar1966} or EPDiff \cite{HoMa2005} equation and reads as 
\begin{equation*}
m_t +\nabla_u m+\g(\nabla u,m)+(\on{div} u) m=0,
\qquad  
m:= P_{\Id} u,
\qquad
u:=\varphi_t\circ\varphi^{-1}.
\end{equation*}
In Lagrangian coordinates, the equation takes the form shown in the following corollary. 
The conditions for local well-posedness in this corollary agree with the ones in \cite{BBCEK2019}, where metrics governed by a general class of pseudo-differential operators are investigated.
The proof is an application of  \autoref{thm:wellposed} to $\Diff(M)$, seen as an open subset of $\Imm(M,M)$.
Moreover, the proof extends \autoref{thm:wellposed} to lower Sobolev regularity using some cancellations which are due to the vanishing normal bundle.
The notation is as in \autoref{thm:wellposed}.

\begin{corollary}[Diffeomorphisms]
\label{cor:diffeos}
A smooth curve $\varphi \colon [0,1]\to \Diff(M)$ is a critical point of the energy functional 
\begin{equation*}
E(\varphi)
=
\frac12 \int_0^1 \int_M \g(P_\varphi \varphi_t, \varphi_t)\vol^g dt
\end{equation*}
if and only if it satisfies the geodesic equation 
\begin{align*}
\nabla_{\p_t} \varphi_t= &P_\varphi\i\Big(-T\varphi\,\g(P_\varphi \varphi_t,\nabla \varphi_t)^\sharp - (\nabla_{\varphi_t}P)\varphi_t -\Tr^g\big(\g(\nabla \varphi_t,T\varphi)\big) P_\varphi \varphi_t\Big).
\end{align*}
The geodesic equation is well-posed in the sense of~\autoref{thm:wellposed} if $P$ satisfies conditions \ref{sec:conditions:a}--\ref{sec:conditions:c} of \autoref{sec:conditions} for some  $p \in [1/2,\infty)$ and all $r\in [r_0,\infty)$ with $r_0 \in (\dim(M)/2+1,\infty)$. 
In particular, this is the case if $P=(1+\Delta)^p$ with
\begin{itemize}
\item $p \in [1,\infty)$ and $r \in (\operatorname{dim}(M)/2+1,\infty)\cap[p+1,\infty)$; or
\item $p \in [1/2,1)$ and $r \in (\operatorname{dim}(M)/2+1,\infty)\cap[p+3/2,\infty)$.
\end{itemize} 
\end{corollary}

\begin{proof}
The formula for the geodesic equation follows from \autoref{thm:geodesic} because the terms $\Adj(\nabla P)^\bot$ and $\nabla Tf=(\nabla Tf)^\bot$ vanish.
To show well-posedness of the geodesic equation, 
note that condition~\ref{sec:conditions:d} of \autoref{sec:conditions} is trivially satisfied because $\Adj(\nabla P)^\bot$ vanishes.
Moreover, note that the condition $p \in [1,\infty)$ in \autoref{thm:wellposed} can be replaced by the weaker condition $p \in [1/2,\infty)$ because the term $\nabla Tf$, which is of second order in $f$, vanishes. 
This can be seen by a term-by-term investigation of the right-hand side of the geodesic equation as in the proof of \autoref{thm:wellposed}.
Therefore, the geodesic equation is well-posed for any operator field $P$ satisfying conditions \ref{sec:conditions:a}--\ref{sec:conditions:c} of \autoref{sec:conditions} for some  $p \in [1/2,\infty)$ and all $r\in [r_0,\infty)$ with $r_0 \in (\dim(M)/2+1,\infty)$, as claimed. 

It remains to verify these conditions for the specific operator $P=(1+\Delta)^p$. 
Condition~\ref{sec:conditions:a} for $p\geq 1$ follows from \autoref{lem:dependence},
and condition~\ref{sec:conditions:a} for $p\in[1/2,1)$ is verified as follows.
We split the operator $P_\varphi$ in two components, 
$$P_\varphi=(1+\De^{\varphi^*\g})^{-1}(1+\De^{\varphi^*\g})^{1+p}.$$
As $1+p\geq 1$, \autoref{lem:dependence} shows that the operator $(1+\De^{\varphi^*\g})^{1+p}$ is a real analytic section of the bundle
$$GL(H^r_{\Diff^r}(M,TM),H^{r-2p-2}_{\Diff^r}(M,TM)) \to \Diff^r(M)$$
for any $r$ such that $r-2p-2\geq 1-r$, i.e., $r\geq p+3/2$.
Similarly, under even weaker conditions, the operator $(1+\De^{\varphi^*\g})^{-1}$ is a real analytic section of the bundle 
$$GL(H^{r-2p-2}_{\Diff^r}(M,TM),H^{r-2p}_{\Diff^r}(M,TM)) \to \Diff^r(M).$$
By the chain rule, the operator $P_\varphi$ is real analytic as required in condition~\ref{sec:conditions:a}.
Conditions~\ref{sec:conditions:b} and \ref{sec:conditions:c} can be verified as in the proof of~\autoref{thm:satisfies}.
\end{proof}

Next we consider reparametrization-invariant Sobolev metrics on spaces of immersed curves, i.e., we consider the special case $M=S^1$. 
Our interest in these spaces stems from their fundamental role in the field of mathematical shape analysis; see e.g.~\cite{bauer2014overview,younes1998computable,klassen2004analysis,sundaramoorthi2011new,bauer2017numerical}
for $\mathbb R^n$-valued curves and \cite{le2017computing,su2014statistical,celledoni2016shape,su2018comparing} for manifold-valued curves. 
For curves in $\mathbb R^n$ local well-posedness of the geodesic equation for integer-order metrics has been shown in \cite{michor2007overview}. 
This has recently been extended to fractional-order metrics in \cite{bauer2018fractional}. 
The following corollary of our main result further generalizes this to fractional-order metrics on spaces of manifold-valued curves:

\begin{corollary}[Curves]
\label{cor:curves}
A smooth curve $c \colon [0,1]\to \Imm(S^1,N)$ is a critical point of the energy functional 
\begin{equation*}
E(c)
=
\frac12 \int_0^1 \int_M \g(P_c c_t, c_t)|\partial_{\theta} c| d\theta dt
\end{equation*}
if and only if it satisfies the geodesic equation 
\begin{align*}
\nabla_{\p_t} c_t= &\frac12P_c\i\Big(\Adj(\nabla P)_c(c_t,c_t)^\bot-2\,\g(P_c c_t,\nabla_{\partial_s} c_t) v_c
-\g(P_c c_t,c_t)\,H_c \Big)
\\&
-P_c\i\Big((\nabla_{c_t}P)c_t+\big(\g(\nabla_{\partial_s} c_{t},v_c)\big) P_c c_t\Big)\;,
\end{align*}
where $\partial_s=|c_{\theta}|^{-1}\partial_{\theta}$ denotes the normalization of the coordinate vector field $\partial_{\theta}$, 
$v_c={\partial_s}c$ the unit-length tangent vector,
and $H_c=(\nabla_{\partial_s} v_c)^{\bot}$ the vector-valued curvature of $c$. 

If the operator $P$ satisfies the conditions of \autoref{sec:conditions} for some  $p \in [1,\infty)$ and all $r\in [r_0,\infty)$ with $r_0 \in (\dim(M)/2+1,\infty)$, then the geodesic equation is well-posed in the sense of~\autoref{thm:wellposed}. This is in particular the case for the operator 
$P=(1-\nabla_{\partial_s}\nabla_{\partial_s})^p$ if $p \in [1,\infty)$ and $r \in (\operatorname{dim}(M)/2+1,\infty)\cap[p+1,\infty)$.
\end{corollary}
\begin{proof}
This follows directly from~\autoref{thm:geodesic}, \autoref{thm:satisfies} and \autoref{thm:wellposed}.
\end{proof}

The last special case to be discussed in this section is $N=\mathbb R^n$, which includes in particular the space of surfaces in $\mathbb R^3$. In the article~\cite{Bauer2011b} we proved a local well-posedness result for integer-order metrics. The proof given there had a gap, which has been corrected in the article~\cite{Mueller2017}. The following corollary of our main result extends this to fractional order metrics:

\begin{corollary}[Flat ambient space]
\label{cor:flat}
A smooth curve $f \colon [0,1]\to \Imm(M,\mathbb R^n)$ is a critical point of the energy functional 
\begin{equation*}
E(f)
=
\frac12 \int_0^1 \int_M \langle P_f f_t, f_t\rangle \vol^g dt
\end{equation*}
if and only if it satisfies the geodesic equation 
\begin{align*}
 f_{tt}= &\frac12P_f\i\Big(\Adj(dP)_f(f_t,f_t)^\bot-2df\,\langle P_f f_t, df_t\rangle^\sharp
-\langle P_f f_t,f_t\rangle\,H_f\Big)
\\&
-P_f\i\Big((\nabla_{f_t}P)f_t+\Tr^g\big(\langle d f_t,df\rangle\big) P_ff_t\Big),
\end{align*}
where $\langle\cdot,\cdot\rangle$ denotes the Euclidean scalar product on $\mathbb R^n$, $g=f^*\langle\cdot,\cdot\rangle$ the induced pullback metric on $M$, and   $H_f=\Tr^g(d^2f)^{\bot}$ the vector-valued mean curvature of $f$. 

If the operator $P$ satisfies the conditions of \autoref{sec:conditions} for some  $p \in [1,\infty)$ and all $r\in [r_0,\infty)$ with $r_0 \in (\dim(M)/2+1,\infty)$, then the geodesic equation is well-posed in the sense of~\autoref{thm:wellposed}. This is in particular the case for the operator 
$P=(1+\Delta)^p$ with $p \in [1,\infty)$ and $r \in (\operatorname{dim}(M)/2+1,\infty)\cap[p+1,\infty)$.
\end{corollary}
\begin{proof}
This follows from Theorems~\ref{thm:geodesic}, \ref{thm:wellposed}, and \ref{thm:satisfies} with $N=\mathbb R^n$, noting that the covariant derivative on $\mathbb R^n$ and the induced covariant derivative on $\Imm^r(M,\mathbb R^n)$ coincide with ordinary derivatives.
\end{proof}

\appendix
\section{The push-forward operator on Sobolev spaces}
\begin{theorem}[Smooth curves in convenient vector spaces] {\rm \cite[~4.1.19]{FK88}}
\label{thm:FK}
Let $c\colon\mathbb R\to E$ be a curve in a convenient vector space $E$. Let 
$\mathcal{V}\subset E'$ be a subset of bounded linear functionals such that 
the bornology of $E$ has a basis of $\sigma(E,\mathcal{V})$-closed sets. 
Then the following are equivalent:
\begin{enumerate}
\item $c$ is smooth
\item For each $k\in\mathbb N$ there exists a locally bounded curve $c^{k}\colon\mathbb R\to E$ such that for each $\ell\in\mathcal V$ the function $\ell\circ c$ is smooth $\mathbb R\to \mathbb R$ with $(\ell\circ c)^{(k)}=\ell\circ c^{k}$. 
\end{enumerate}
If $E$ is reflexive, then for any point separating subset
$\mathcal{V}\subset E'$ the bornology of $E$ has a basis of 
$\si(E,\mathcal{V})$-closed subsets, by {\rm \cite[~4.1.23]{FK88}}.
\end{theorem}
This theorem is surprisingly strong: 
Note that $\mathcal V$ does not need to recognize bounded sets. 
We shall use the theorem in situations where $\mathcal V$ is just the set of all point evaluations on suitable Sobolev spaces.

\begin{lemma}[Smooth curves in Sobolev spaces of sections]
\label{lem:curvesSobolev}
Let $E$ be a vector bundle over $M$, and let $\nabla$ be a connection on $E$. Then it holds for each $r\in(\dim(M)/2,\infty)$ that the space $C^\infty(\mathbb R,\Ga_{H^r}(E))$ of smooth curves in $\Ga_{H^r}(E)$ consists of all continuous mappings $c\colon\mathbb R\x M \to E$ with $p\circ c = \on{pr}_2\colon\mathbb R\x M\to M$ such that: 
\begin{itemize}
 \item For each $x\in M$ the curve $t\mapsto c(t,x)\in E_x$ is smooth; 
        let $(\p^p_t c)(t,x) = \p_t^p(c(t,x))$, and
 \item For each $p\in \mathbb N_{\ge0}$, the curve $\p_t^pc$ has values in $\Ga_{H^r}(E)$ 
       so that $\p_t^pc \colon\mathbb R\to \Ga_{H^r}(E)$, and 
       $t \mapsto \|\p_t c(t,\quad)\|_{H^r}$ is bounded, 
       locally in $t$.
\end{itemize}
\end{lemma}

\begin{proof}
To see this we first choose a second vector bundle $F\to M$ such that $E\oplus_M F$ is a trivial bundle, i.e., isomorphic to $M\x \mathbb R^n$ for some $n\in\mathbb N$. 
Then $\Ga_{H^r}(E)$ is a direct summand in $H^r(M,\mathbb R^n)$, so that we may assume without loss that $E$ is a trivial bundle, and then, that it is 1-dimensional. 
So we have to identify $C^\infty(\mathbb R,H^r(M,\mathbb R))$. 
But in this situation we can just apply Theorem \ref{thm:FK} for the set $\mathcal V\subset H^s(M,\mathbb R)'$ consisting just of all point evaluations $\on{ev}_x\colon H^r(M,\mathbb R)\to \mathbb R$.
\end{proof}

\begin{lemma}[Function spaces of mixed smoothness]
\label{lem:mixed}
Let $U$ be an open subset of a finite-dimensional vector space, 
let $r \in (\dim(M)/2,\infty)$,
let $\al\in\{\infty,\om\}$, 
and let $C^\al(U)=\varprojlim_p E_p$ be the representation of the complete locally convex space $C^\al(U)$ as a projective limit of Banach spaces $E_p$.
Then
\begin{equation*}
H^rC^\al(M\times U)
:= 
C^\al(U,H^r(M))
=
H^r(M)\hat\otimes C^\al(U)
=
H^r(M,C^\al(U)),
\end{equation*}
where $\hat\otimes$ is the injective, projective, or bornological tensor product, or any tensor product in-between,
and where $H^r(M,C^\al(U))$ is defined as the projective limit $\varprojlim_p H^r(M,E_p)$.
\end{lemma}

The lemma justifies the following notation, which shall be used in \autoref{lem:pushforward} below. 
If $E_1$ and $E_2$ are vector bundles over $M$, 
and $U \subseteq E_1$ is an open neighborhood of the image of an $H^r$ section,
then we write $\Gamma_{H^r}(C^\al(U,E_2))$ for the set of all fiber-preserving functions $F\colon U \to E_2$ which have regularity $H^rC^\al$ in every $C^\al$ vector bundle chart of $E_1$.
Loosely speaking, these are sections of regularity $H^r$ in the foot point and regularity $C^\al$ in the fibers.

\begin{proof}
The space $C^\infty(U)$ is nuclear by \cite[Corollary to Theorem~51.4]{treves1967topological}, 
and the space $C^\om(U)$ is nuclear as a countable inductive limit of nuclear spaces of holomorphic functions \cite[Theorem~30.11]{KM97}.
Let $\otimes_\ep$, $\otimes_\pi$, and $\otimes_\be$ be the injective, projective, and bornological completed tensor products, respectively. 
Then 
\begin{align*}
C^\al(U) \otimes_\ep H^r(M)
=
C^\al(U) \otimes_\pi H^r(M)
=
C^\al(U) \otimes_\be H^r(M),
\end{align*}
where the first equality holds because $C^\al(U)$ is nuclear, 
and the second equality holds by \cite[Proposition~5.8]{KM97} using that $H^r(M)$ is a normed space, and $C^\omega(V)$ is an (LF)-space and therefore bornological.
Thus, all tensor spaces $C^\al(U)\hat\otimes H^r(M)$ are equal. 
Moreover, 
\begin{align*}
C^\infty(U,H^r(M))
=
C^\infty(U) \otimes_\ep H^r(M)
\end{align*}
by \cite[Theorem~44.1]{treves1967topological}, and
\begin{align*}
C^\om(U,H^r(M))
=
\varprojlim_{\tilde U} \mathcal H(\tilde U,H^r(M))
=
\varprojlim_{\tilde U} \mathcal H(\tilde U)\hat\otimes H^r(M)
=
C^\om(U) \hat\otimes H^r(M)
\end{align*}
by \cite[Corollary~16.7.5]{jarchow2012locally}, where $\mathcal H$ denotes holomorphic functions and $\tilde U$ are open neighborhoods of $U$ in the complexification of the underlying vector space.
Let $\De_2$ be the natural norm on $L^2$ functions \cite[7.1]{defant1992tensor}.
Then
\begin{equation*}
H^r(M) \hat\otimes C^\al(U) 
=
H^r(M) \hat\otimes_{\De_2} C^\al(U)  
=
\varprojlim_p H^r(M) \hat\otimes_{\De_2} E_p 
=
\varprojlim_p H^r(M, E_p),
\end{equation*}
where the first equality holds because $\ep\le\De_2\le\pi$ \cite[7.1]{defant1992tensor},
the second one by the definition of tensor products of locally convex spaces \cite[35.2]{defant1992tensor},
and the third one because the fractional Laplacian $(1+\De^g)\colon H^r(M)\to L^2(M)$ with respect to any auxiliary Riemannian metric $g \in \Met(M)$ is an isometry and because $L^2(M,E_p)=L^2(M)\otimes_{\De_2}E_p$ by the definition of $\De_2$ \cite[7.2]{defant1992tensor}. 
\end{proof}

\begin{lemma}[Push-forward of functions]
\label{lem:omega}
Let $U$ be an open subset of $\mathbb R$,
and let $r \in (\dim(M)/2,\infty)$. 
Then $H^r(M,U)$ is open in $H^r(M,\mathbb R)$, 
and the following statements hold.
\begin{enumerate}
\item 
\label{lem:omega:a}
The following map is smooth:
\begin{align*}
H^rC^\infty(M\times U) \times H^r(M,U) \ni (F,h) \mapsto F\circ (\Id_M,h) \in H^r(M).
\end{align*}

\item 
\label{lem:omega:b}
The following map is real analytic:
\begin{align*}
H^rC^\omega(M\times U) \times H^r(M,U) \ni (F,h) \mapsto F\circ (\Id_M,h) \in H^r(M).
\end{align*}
\end{enumerate}
\end{lemma}

\begin{proof}
The set $\Ga_{H^r}(U)$ is open in $\Ga_{H^r}(E_1)$ because $\Ga_{H^r}(E_1)$ is continuously included in $\Ga_{C}(E_1)$ thanks to the Sobolev embedding theorem.
\begin{enumerate}[wide]
\item follows from the more general statement \autoref{lem:pushforward}.\ref{lem:pushforward:a}.

\item[\textbf{(b')}]
\makeatletter
\protected@edef\@currentlabel{\textbf{(b')}}
\makeatother
\label{lem:omega:b'} 
As an intermediate step, we claim that the following map is real analytic:
\begin{align*}
C^\om(U) \times H^r(M,U) \ni (f,h) \mapsto f\circ h \in H^r(M).
\end{align*}
For any $f \in C^\om(U)$ and $h \in H^r(M,U)$, the composition $f\circ h$ coincides with the Riesz functional calculus $f(h)$, which is defined as follows \cite[Theorem~4.7]{conway2013course}.
As the spectrum $\sigma(h)$ equals the range of $h$, which is a compact subset of $U$, 
there is a set of positively oriented curves $\Gamma=\{\gamma_1,\dots,\gamma_n\}$ in $U \setminus \sigma(h)$ such that $\sigma(h)$ is inside of $\Gamma$, and $\mathbb C\setminus U$ is outside of $\Gamma$ \cite[Proposition~4.4]{conway2013course}. 
Then one defines $f(h)$ as the following Bochner integral over the resolvent of $h$:
\begin{align*}
f(h) 
=
\frac{-1}{2\pi \mathrm{i} }\int_\Gamma f(\la) (h-\la)^{-1} d\la
\end{align*}
For any fixed $\Gamma$, this integral is well-defined and real analytic as claimed.

\item 
The following map is real analytic thanks to \autoref{lem:omega:b'} and the boundedness of multiplication $H^r(M)\times H^r(M)\to H^r(M)$:
\begin{align*}
H^r(M) \times C^\omega(U) \times H^r(M,U) \ni (a,f,h) \mapsto (a\otimes f)\circ (\Id_M,h) \in H^r(M),
\end{align*}
where $(a\otimes f)\circ(\Id_M,h)$ denotes the map $x \mapsto a(x)f(h(x))$.
Equivalently, by the real analytic exponential law \cite[11.18]{KM97}, the following map is real analytic: 
\begin{align*}
H^r(M) \times C^\omega(U) \ni (a,f) \mapsto \big(h \mapsto (a\otimes f)\circ(\Id_M, h)\big) \in C^\om(H^r(M,U),H^r(M)).
\end{align*}
This map is bilinear and real analytic, and therefore bounded. 
By the universal property of the bornological tensor product $\otimes_\beta$ \cite[5.7]{KM97}, it descends to a bounded linear map
\begin{align*}
H^r(M) \otimes_\beta C^\omega(U) \ni F \mapsto \big(h \mapsto F\circ (\Id_M,h)\big) \in C^\om(H^r(M,U),H^r(M)).
\end{align*}
The domain of this map equals $H^rC^\om(M\times U)$ by \autoref{lem:mixed}. 
\qedhere
\end{enumerate}
\end{proof}

\begin{lemma}[Push-forward of sections]
\label{lem:pushforward} 
Let $E_1,E_2$ be vector bundles over $M$, 
let $U\subset E_1$ be an open neighborhood of the image of a smooth section, 
let $F\colon U\to E_2$ be a fiber preserving function,  
and let $r \in (\dim(M)/2,\infty)$. 
Then $\Ga_{H^r}(U)$ is open in $\Ga_{H^r}(E_1)$,
and the following statements hold:
\begin{enumerate}
\item
\label{lem:pushforward:a}
If $F$ is smooth or belongs to $\Ga_{H^r}(C^\infty(U,E_2))$, then the push-forward $F_*$ is smooth: 
\begin{equation*}
F_*\colon\Ga_{H^r}(U) \to \Ga_{H^r}(E_2),\quad h\mapsto  F\circ h.
\end{equation*}

\item
\label{lem:pushforward:b}
If $F$ is real analytic or belongs to $\Ga_{H^r}(C^\om(U,E_2))$, then the pushforward $F_*$ is real analytic.
\end{enumerate}
The notation $\Ga_{H^r}(C^\infty(U,E_2))$ and $\Ga_{H^r}(C^\om(U,E_2))$ is explained in Section~\ref{lem:mixed}.
\end{lemma}

\begin{proof} 
\begin{enumerate}[wide]
\item 
Let $c\colon\mathbb R\ni t\mapsto c(t,\cdot)\in \Ga_{H^r}(U)$ be a smooth curve. 
As $r>\dim(M)/2$, it holds for each $x\in M$ that the mapping $\mathbb R\ni t \mapsto F_x(c(t,x))\in (E_2)_x$ is smooth. 
By the Fa\`a di Bruno formula (see \cite{FaadiBruno1855} for the 1-dimensional version, preceded in \cite{Arbogast1800} by 55 years), we have for each $p\in\mathbb N_{>0}$, $t \in \mathbb R$, and $x \in M$ that
\begin{align*}
\p_t^p F_x (c(t,x)) = 
\sum_{j\in\mathbb N_{>0}} \sum_{\substack{\al\in \mathbb N_{>0}^j\\ \al_1+\dots+\al_j =p}}
\frac{1}{j!}d^j (F_x) (c(t,x))\Big(
\frac{\p_t^{(\al_1)}c(t,x)}{\al_1!},\dots,
\frac{\p_t^{(\al_j)}c(t,x)}{\al_j!}\Big)\,.
\end{align*}
For each $x\in M$ and $\al_x\in (E_2)_x^*$ the mapping $s\mapsto \langle s(x),\al_x\rangle$ is a continuous linear functional on the Hilbert space $\Ga_{H^r}(E_2)$. 
The set $\mathcal V_2$ of all of these functionals separates points and therefore satisfies the condition of Theorem~\ref{thm:FK}. We also have for each $p\in\mathbb N_{>0}$, $t \in \mathbb R$, and $x \in M$ that
\begin{align*}
\p_t^p\langle F_x (c(t,x)),\al_x\rangle &= \langle\p_t^p F_x (c(t,x)),\al_x\rangle.
\end{align*}
Using the explicit expressions for $\p_t^p F_x (c(t,x))$ from above we may apply Lemma~\ref{lem:curvesSobolev} to conclude that $t\mapsto F(c(t,\;))$ is a smooth curve $\mathbb R\to \Ga_{H^r}(E_2)$. 
Thus, $F_*$ is a smooth mapping, and we have shown \ref{lem:pushforward:a}.

\item[\textbf{(b')}]
\makeatletter
\protected@edef\@currentlabel{\textbf{(b')}}
\makeatother
\label{lem:pushforward:b'}
We claim that \ref{lem:pushforward:b} holds when $F$ is fiber-wise linear. 
Then $F$ can be identified with a map in $\check F \in \Ga_{H^r}(L(E_1,E_2))$. 
For any $h \in \Gamma_{H^r}(E_1)$, the composition  $F\circ h$ equals the trace $\check F.h$, which is real analytic in $h$ by the module properties \ref{thm:module}.

\item
To prove the general case, we write $E_1$ and $E_2$ as sub-bundles of a trivial bundle $M\times V$. 
The corresponding inclusion and projection mappings are real analytic mappings of vector bundles and are denoted by
\begin{align*}
i_1&\colon E_1 \to M\times V, 
&
i_2&\colon E_2 \to M\times V, 
&
\pi_1&\colon M\times V\to E_1,
&
\pi_2&\colon M\times V\to E_2.
\end{align*} 
Then the set $\tilde U := \pi_1^{-1}(U)\subseteq M\times V$ and the map $\tilde F:=i_2\circ F\circ\pi_1$ fit into the following commutative diagrams:
\begin{equation*}
\xymatrix@C=3em{
U 
\ar[r]^F
\ar@{^(->}[d]^{i_1}
&
E_2
\ar@{<<-}[d]^{\pi_2}
\\
\tilde U
\ar[r]^{\tilde F}
&
M\times V
}
\hspace{4em}
\xymatrix@C=3em{
\Ga_{H^r}(U) 
\ar[r]^{F_*}
\ar@{^(->}[d]^{(i_1)_*}
&
\Ga_{H^r}(E_2)
\ar@{<<-}[d]^{(\pi_2)_*}
\\
\Ga_{H^r}(\tilde U)
\ar[r]^{\tilde F_*}
&
\Ga_{H^r}(M\times V)
}
\end{equation*}
All maps in the diagram on the left are real analytic by definition.
The map $(\tilde F)_*$ is real analytic by \autoref{lem:omega}.\ref{lem:omega:b} applied component-wise to the trivial bundle $M\times V$, 
and the maps $(i_1)_*$ and $(\pi_2)_*$ are real analytic by \ref{lem:pushforward:b'}.
Therefore, $F_*=(\pi_2)_*\circ (\tilde F)_*\circ (i_1)_*$ is real analytic, which proves \ref{lem:pushforward:b}.
\qedhere
\end{enumerate}
\end{proof}

\section{A real analytic no-loss no-gain result}

The following lemma is a variant of the no-loss-no-gain theorem of Ebin and Marsden \cite{EM1970}, adapted to the real analytic sprays on spaces of immersions as in the setting of \autoref{thm:wellposed}.
The proof is a minor adaptation of the proof in \cite{EM1970}; see also \cite{bruveris2017regularity}.

\begin{lemma}[Real analytic no-loss no-gain]
\label{lem:nolossnogain}
Let $r_0>\dim(M)/2+1$ and let $\al\in\{\infty,\om\}$.
For each $r\ge r_0$, 
let $S^r$ be a $\Diff(M)$-invariant $C^\al$ vector field on $T\Imm^r(M,N)$ such that  
$Ti_{r,s} \circ S^r = S^s \circ i_{r,s}$ where $i_{r,s}\colon T\Imm^r(M,N)\to T\Imm^s(M,N)$
is the $C^\al$-embedding for $r_0\le s<r$.
By the theorem of Picard-Lindel\"of each $S^r$ has a maximal $C^\al$-flow 
$\on{Fl}^{S^r}\colon U^r \to T\Imm^r(M,N)$ for an open neighborhood $U^r$ of $\{0\}\x T\Imm^r(M,N)$ in 
$\mathbb R\x T\Imm^r(M,N)$.

Then $U^r = U^s\cap( \mathbb R\x T\Imm^r(M,N))$ for all $r_0+1 \le r$ and $r_0\le s\leq r$. 
Thus, there is no loss or gain in regularity during the evolution along any $S^r$ for $r\ge r_0+1$.
\end{lemma}

\begin{proof} 
\begin{enumerate}[wide]
\item\label{lem:nolossnogain:result}
We shall use the following result \cite[Lemma~12.2]{EM1970}:
\emph{Any $h\in H^r(M,TN)$ such that $Th\circ X\in H^r(M,TTN)$ for all $X\in \X(M)$ satisfies $h\in H^{r+1}(M,TN)$.}

\item\label{lem:nolossnogain:J}
For $h\in T\Imm^r(M,N)$ let $J^r_h$ be the open interval such that $U^r\cap (\mathbb R\x \{h\}) = J^r_h\x \{h\}$, i.e., $J^r_h$ is the maximal domain of the integral curve of $S^r$ through $h$ in $T\Imm^r(M,N)$; see \cite[32.14]{KM97}. Since $i_{r,s}\circ \on{Fl}^{S^r}_t = \on{Fl}^{S^s}_t\circ$ 
(see \cite[32.16]{KM97}), for $h\in T\Imm^r(M,N)$ we have $J^r_h\subseteq J^s_h$ for $r_0\le s<r$.

\item\label{lem:nolossnogain:claim}
{\bf Claim.} \emph{For $h\in T\Imm^{r+1}(M,N)$ we have $J^r_h = J^{r+1}_h$.}
\\
Since $S^r$ is invariant under the pullback action of $\Diff(M)$, we have for $h\in T\Imm^{r+1}(M,N)$ and any $X\in \X(M)$ that
$$
\on{Fl}^{S^r}_t(h\circ \on{Fl}^X_u) = \on{Fl}^{S^r}_t(h) \circ \on{Fl}^X_u\,.
$$ 
Differentiating both side we get 
\begin{align*}
T(\on{Fl}^{S^r}_t(h))\circ X &= \p_u|_0 ( \on{Fl}^{S^r}_t(h) \circ \on{Fl}^X_u) =
\p_u|_0 (\on{Fl}^{S^r}_t(h\circ \on{Fl}^X_u)) 
\\
&=  T(\on{Fl}^{S^r}_t)(\p_u|_0(h\circ \on{Fl}^X_u))
= T(\on{Fl}^{S^r}_t)(Th\circ X)
\end{align*}
Since $Th\circ X \in H^r(M,TTN)$ we see that $T(\on{Fl}^{S^r}_t(h))\circ X  \in H^r(M,TTN)$. By result \ref{lem:nolossnogain:result} we get
$\on{Fl}^{S^r}_t(h)\in T\Imm^{r+1}(M,N)$, and thus $J^{r}_h\supseteq J^{r+1}_h$. The converse inclusion is \ref{lem:nolossnogain:J}.

\item
Let $r_0+1\le s < r<s+1$ and let $h\in T\Imm^{r}(M,N)$. Then 
$$J^r_h \subseteq J^s_h\subseteq J^{r-1}_h= J^r_h,$$ 
where the inclusions follow from \ref{lem:nolossnogain:J}, \ref{lem:nolossnogain:J}, and \ref{lem:nolossnogain:claim}, respectively.
Thus we have $J^r_h=J^s_h=J^{r-1}_h$.
\qedhere 
\end{enumerate}
\end{proof}

\bibliographystyle{abbrv}

\begin{thebibliography}{10}

\bibitem{Arbogast1800}
L.~F.~A. Arbogast.
\newblock {\em Du calcul des d\'erivations}.
\newblock Levrault, Strasbourg, 1800.

\bibitem{Ar1966}
V.~I. Arnold.
\newblock Sur la g\'eom\'etrie diff\'erentielle des groupes de {L}ie de
  dimension infinie et ses applications \`a l'hydrodynamique des fluides
  parfaits.
\newblock {\em Ann. Inst. Fourier (Grenoble)}, 16(fasc. 1):319--361, 1966.

\bibitem{BBCEK2019}
M.~Bauer, M.~Bruveris, E.~Cismas, J.~Escher, and B.~Kolev.
\newblock Well-posedness of the {EPDiff} equation with a pseudo-differential
  inertia operator.
\newblock To appear in Journal of Differential Equations, 2020.

\bibitem{bauer2012vanishing}
M.~Bauer, M.~Bruveris, P.~Harms, and P.~W. Michor.
\newblock Vanishing geodesic distance for the {R}iemannian metric with geodesic
  equation the {K}d{V}-equation.
\newblock {\em Annals of Global Analysis and Geometry}, 41(4):461--472, 2012.

\bibitem{bauer2013geodesic}
M.~Bauer, M.~Bruveris, P.~Harms, and P.~W. Michor.
\newblock Geodesic distance for right invariant {S}obolev metrics of fractional
  order on the diffeomorphism group.
\newblock {\em Annals of Global Analysis and Geometry}, 44(1):5--21, 2013.

\bibitem{bauer2018smooth}
M.~Bauer, M.~Bruveris, P.~Harms, and P.~W. Michor.
\newblock Smooth perturbations of the functional calculus and applications to
  {R}iemannian geometry on spaces of metrics.
\newblock {\em {\normalfont\textrm{arXiv:}\texttt{1810.03169}}}, 2018.

\bibitem{bauer2017numerical}
M.~Bauer, M.~Bruveris, P.~Harms, and J.~M{\o}ller-Andersen.
\newblock A numerical framework for {Sobolev} metrics on the space of curves.
\newblock {\em SIAM Journal on Imaging Sciences}, 10(1):47--73, 2017.

\bibitem{bauer2018fractional}
M.~Bauer, M.~Bruveris, and B.~Kolev.
\newblock Fractional {Sobolev} metrics on spaces of immersed curves.
\newblock {\em Calculus of Variations and Partial Differential Equations},
  57(1):27, 2018.

\bibitem{bauer2014overview}
M.~Bauer, M.~Bruveris, and P.~W. Michor.
\newblock Overview of the geometries of shape spaces and diffeomorphism groups.
\newblock {\em Journal of Mathematical Imaging and Vision}, 50(1-2):60--97,
  2014.

\bibitem{bauer2015local}
M.~Bauer, J.~Escher, and B.~Kolev.
\newblock Local and global well-posedness of the fractional order {EPD}iff
  equation on {$\mathbb R^d$}.
\newblock {\em Journal of Differential Equations}, 258(6):2010--2053, 2015.

\bibitem{Bauer2011b}
M.~Bauer, P.~Harms, and P.~W. Michor.
\newblock Sobolev metrics on shape space of surfaces.
\newblock {\em J. Geom. Mech.}, 3(4):389--438, 2011.

\bibitem{bauer2012sobolev}
M.~Bauer, P.~Harms, and P.~W. Michor.
\newblock Sobolev metrics on shape space, {II}: weighted {S}obolev metrics and
  almost local metrics.
\newblock {\em Journal of Geometric Mechanics}, 4(4), 2012.

\bibitem{bauer2018vanishing}
M.~Bauer, P.~Harms, and S.~C. Preston.
\newblock Vanishing distance phenomena and the geometric approach to {SQG}.
\newblock {\em Archive for Rational Mechanics and Analysis}, 2019.

\bibitem{bauer2016geometric}
M.~Bauer, B.~Kolev, and S.~C. Preston.
\newblock Geometric investigations of a vorticity model equation.
\newblock {\em Journal of Differential Equations}, 260(1):478--516, 2016.

\bibitem{behzadan2017certain}
A.~Behzadan and M.~Holst.
\newblock On certain geometric operators between {S}obolev spaces of sections
  of tensor bundles on compact manifolds equipped with rough metrics.
\newblock {\em {\normalfont\textrm{arXiv:}\texttt{1704.07930}}}, 2017.

\bibitem{bruveris2017regularity}
M.~Bruveris.
\newblock Regularity of maps between {Sobolev} spaces.
\newblock {\em Annals of Global Analysis and Geometry}, 52(1):11--24, 2017.

\bibitem{camassa1993integrable}
R.~Camassa and D.~D. Holm.
\newblock An integrable shallow water equation with peaked solitons.
\newblock {\em Physical Review Letters}, 71(11):1661, 1993.

\bibitem{celledoni2016shape}
E.~Celledoni, S.~Eidnes, and A.~Schmeding.
\newblock Shape analysis on homogeneous spaces: a generalised {SRVT} framework.
\newblock In {\em The Abel Symposium}, pages 187--220. Springer, 2016.

\bibitem{constantin1985simple}
P.~Constantin, P.~D. Lax, and A.~Majda.
\newblock A simple one-dimensional model for the three-dimensional vorticity
  equation.
\newblock {\em Communications on pure and applied mathematics}, 38(6):715--724,
  1985.

\bibitem{constantin1994formation}
P.~Constantin, A.~J. Majda, and E.~Tabak.
\newblock Formation of strong fronts in the {2-D} quasigeostrophic thermal
  active scalar.
\newblock {\em Nonlinearity}, 7(6):1495, 1994.

\bibitem{conway2013course}
J.~B. Conway.
\newblock {\em A course in functional analysis}, volume~96.
\newblock Springer Science \& Business Media, 2013.

\bibitem{defant1992tensor}
A.~Defant and K.~Floret.
\newblock {\em Tensor norms and operator ideals}, volume 176.
\newblock Elsevier, 1992.

\bibitem{EM1970}
D.~G. Ebin and J.~E. Marsden.
\newblock {G}roups of diffeomorphisms and the motion of an incompressible
  fluid.
\newblock {\em Ann. of Math.}, 92:102--163, 1970.

\bibitem{escher2014right}
J.~Escher and B.~Kolev.
\newblock Right-invariant {S}obolev metrics of fractional order on the
  diffeomorphism group of the circle.
\newblock {\em Journal of Geometric Mechanics}, 6(3):335--372, 2014.

\bibitem{FaadiBruno1855}
C.~F. Fa\`a~di Bruno.
\newblock Note {s}ur {u}ne {n}ouvelle {f}ormule {d}u {c}alcul
  {d}iff\'erentielle.
\newblock {\em Quart. J. Math.}, 1:359--360, 1855.

\bibitem{FK88}
A.~Fr{\"o}licher and A.~Kriegl.
\newblock {\em Linear spaces and differentiation theory}.
\newblock Pure and Applied Mathematics (New York). John Wiley \& Sons Ltd.,
  Chichester, 1988.

\bibitem{grenander1998computational}
U.~Grenander and M.~I. Miller.
\newblock Computational anatomy: An emerging discipline.
\newblock {\em Quarterly of applied mathematics}, 56(4):617--694, 1998.

\bibitem{grosse2013sobolev}
N.~Gro{\ss}e and C.~Schneider.
\newblock Sobolev spaces on {R}iemannian manifolds with bounded geometry:
  {G}eneral coordinates and traces.
\newblock {\em Mathematische Nachrichten}, 286(16):1586--1613, 2013.

\bibitem{HoMa2005}
D.~D. Holm and J.~E. Marsden.
\newblock Momentum maps and measure-valued solutions (peakons, filaments, and
  sheets) for the {EPD}iff equation.
\newblock In {\em The breadth of symplectic and {P}oisson geometry}, volume 232
  of {\em Progr. Math.}, pages 203--235. Birkh\"auser Boston, Boston, MA, 2005.

\bibitem{hunter1991dynamics}
J.~K. Hunter and R.~Saxton.
\newblock Dynamics of director fields.
\newblock {\em SIAM Journal on Applied Mathematics}, 51(6):1498--1521, 1991.

\bibitem{inci2013regularity}
H.~Inci, T.~Kappeler, and P.~Topalov.
\newblock On the regularity of the composition of diffeomorphisms.
\newblock {\em Mem. Amer. Math. Soc.}, 226(1062):vi+60, 2013.

\bibitem{jarchow2012locally}
H.~Jarchow.
\newblock {\em Locally convex spaces}.
\newblock Springer Science \& Business Media, 2012.

\bibitem{jermyn2017}
I.~Jermyn, S.~Kurtek, H.~Laga, and A.~Srivastava.
\newblock Elastic shape analysis of three-dimensional objects.
\newblock {\em Synthesis Lectures on Computer Vision}, 7:1--185, 09 2017.

\bibitem{jerrard2019vanishing}
R.~L. Jerrard and C.~Maor.
\newblock Vanishing geodesic distance for right-invariant {S}obolev metrics on
  diffeomorphism groups.
\newblock {\em Annals of Global Analysis and Geometry}, 55(4):631--656, 2019.

\bibitem{khesin2008geometry}
B.~Khesin and R.~Wendt.
\newblock {\em The geometry of infinite-dimensional groups}, volume~51.
\newblock Springer Science \& Business Media, 2008.

\bibitem{klassen2004analysis}
E.~Klassen, A.~Srivastava, M.~Mio, and S.~H. Joshi.
\newblock Analysis of planar shapes using geodesic paths on shape spaces.
\newblock {\em IEEE transactions on pattern analysis and machine intelligence},
  26(3):372--383, 2004.

\bibitem{KMS93}
I.~Kol{\'a}{\v{r}}, P.~W. Michor, and J.~Slov{\'a}k.
\newblock {\em Natural operations in differential geometry}.
\newblock Springer-Verlag, Berlin, 1993.

\bibitem{Kol2017}
B.~Kolev.
\newblock Local well-posedness of the {EPD}iff equation: a survey.
\newblock {\em J. Geom. Mech.}, 9(2):167--189, 2017.

\bibitem{kouranbaeva1999camassa}
S.~Kouranbaeva.
\newblock The {C}amassa--{H}olm equation as a geodesic flow on the
  diffeomorphism group.
\newblock {\em Journal of Mathematical Physics}, 40(2):857--868, 1999.

\bibitem{KM96}
A.~Kriegl and P.~W. Michor.
\newblock Product preserving functors of infinite-dimensional manifolds.
\newblock {\em Arch. Math. (Brno)}, 32(4):289--306, 1996.

\bibitem{KM97}
A.~Kriegl and P.~W. Michor.
\newblock {\em The convenient setting of global analysis}, volume~53 of {\em
  Mathematical Surveys and Monographs}.
\newblock American Mathematical Society, Providence, RI, 1997.

\bibitem{le2017computing}
A.~Le~Brigant.
\newblock Computing distances and geodesics between manifold-valued curves in
  the {SRV} framework.
\newblock {\em Journal of Geometric Mechanics}, 9(2):131--156, 2017.

\bibitem{lenells2007hunter}
J.~Lenells.
\newblock The {H}unter--{S}axton equation describes the geodesic flow on a
  sphere.
\newblock {\em Journal of Geometry and Physics}, 57(10):2049--2064, 2007.

\bibitem{marsden1984semidirect}
J.~E. Marsden, T.~Ratiu, and A.~Weinstein.
\newblock Semidirect products and reduction in mechanics.
\newblock {\em Transactions of the american mathematical society},
  281(1):147--177, 1984.

\bibitem{Michor08}
P.~W. Michor.
\newblock {\em Topics in differential geometry}, volume~93 of {\em Graduate
  Studies in Mathematics}.
\newblock American Mathematical Society, Providence, RI, 2008.

\bibitem{Michor20}
P.~W. Michor.
\newblock Manifolds of mappings for continuum mechanics.
\newblock In {\em Geometric {C}ontinuum {M}echanics -- an {O}verview}, pages
  1-- 60. Birkhauser, 2020.

\bibitem{michor2007overview}
P.~W. Michor and D.~Mumford.
\newblock An overview of the {R}iemannian metrics on spaces of curves using the
  {H}amiltonian approach.
\newblock {\em Appl. Comput. Harmon. Anal.}, 23(1):74--113, 2007.

\bibitem{misiolek1998shallow}
G.~Misio{\l}ek.
\newblock A shallow water equation as a geodesic flow on the {B}ott-{V}irasoro
  group.
\newblock {\em Journal of Geometry and Physics}, 24(3):203--208, 1998.

\bibitem{misiolek2010fredholm}
G.~Misio{\l}ek and S.~C. Preston.
\newblock Fredholm properties of {R}iemannian exponential maps on
  diffeomorphism groups.
\newblock {\em Inventiones mathematicae}, 179(1):191, 2010.

\bibitem{Mueller2017}
O.~M\"uller.
\newblock Applying the index theorem to non-smooth operators.
\newblock {\em Journal of Geometry and Physics}, 116:140--145, 2017.

\bibitem{ovsienko1987korteweg}
V.~Y. Ovsienko and B.~A. Khesin.
\newblock {K}orteweg--de {V}ries superequation as an {E}uler equation.
\newblock {\em Functional Analysis and Its Applications}, 21(4):329--331, 1987.

\bibitem{shnirel1987geometry}
A.~I. Shnirel'man.
\newblock On the geometry of the group of diffeomorphisms and the dynamics of
  an ideal incompressible fluid.
\newblock {\em Mathematics of the USSR-Sbornik}, 56(1):79, 1987.

\bibitem{srivastava-klassen-book:2016}
A.~Srivastava and E.~Klassen.
\newblock {\em Functional and Shape Data Analysis}.
\newblock Springer Series in Statistics, 2016.

\bibitem{su2014statistical}
J.~Su, S.~Kurtek, E.~Klassen, A.~Srivastava, et~al.
\newblock Statistical analysis of trajectories on {R}iemannian manifolds: bird
  migration, hurricane tracking and video surveillance.
\newblock {\em The Annals of Applied Statistics}, 8(1):530--552, 2014.

\bibitem{su2018comparing}
Z.~Su, E.~Klassen, and M.~Bauer.
\newblock Comparing curves in homogeneous spaces.
\newblock {\em Differential Geometry and its Applications}, 60:9--32, 2018.

\bibitem{sundaramoorthi2011new}
G.~Sundaramoorthi, A.~Mennucci, S.~Soatto, and A.~Yezzi.
\newblock A new geometric metric in the space of curves, and applications to
  tracking deforming objects by prediction and filtering.
\newblock {\em SIAM Journal on Imaging Sciences}, 4(1):109--145, 2011.

\bibitem{treves1967topological}
F.~Treves.
\newblock {\em Topological Vector Spaces, Distributions and Kernels: Pure and
  Applied Mathematics}.
\newblock Academic Press, 1967.

\bibitem{triebel1992theory2}
H.~Triebel.
\newblock {\em Theory of function spaces {II}}, volume~84 of {\em Monographs in
  Mathematics}.
\newblock Birkh\"auser, 1992.

\bibitem{vishik1978analogs}
S.~Vishik and F.~Dolzhanskii.
\newblock Analogs of the {E}uler--{L}agrange equations and magnetohydrodynamics
  equations related to {L}ie groups.
\newblock In {\em Sov. Math. Doklady}, volume~19, pages 149--153, 1978.

\bibitem{Was2016}
P.~Washabaugh.
\newblock The {SQG} equation as a geodesic equation.
\newblock {\em Archive for Rational Mechanics and Analysis}, 222(3):1269--1284,
  2016.

\bibitem{wunsch2010geodesic}
M.~Wunsch.
\newblock On the geodesic flow on the group of diffeomorphisms of the circle
  with a fractional {S}obolev right-invariant metric.
\newblock {\em Journal of Nonlinear Mathematical Physics}, 17(1):7--11, 2010.

\bibitem{younes1998computable}
L.~Younes.
\newblock Computable elastic distances between shapes.
\newblock {\em SIAM Journal on Applied Mathematics}, 58(2):565--586, 1998.

\bibitem{younes2010shapes}
L.~Younes.
\newblock {\em Shapes and diffeomorphisms}, volume 171.
\newblock Springer, 2010.

\end{thebibliography}

\end{document}